\pdfoutput=1
\documentclass{amsart} 

\usepackage{tikz}
\usetikzlibrary{calc}
\usetikzlibrary{decorations.markings}
\usetikzlibrary {arrows.meta}
\usetikzlibrary {bending}

\usepackage{amssymb}
\usepackage{amsmath}
\usepackage{amsfonts}
\usepackage{url}
\usepackage{enumerate}
\usepackage{graphicx}
\usepackage{mathtools}
\usepackage{mathrsfs}

\usepackage[pdfencoding=unicode]{hyperref}
\hypersetup{
    colorlinks  =   true,
    linkcolor   =   blue,
    filecolor   =   blue,
    urlcolor    =   blue,
    citecolor   =   blue,
  pdfauthor={},
  pdftitle={},
  pdfsubject={},
  pdfkeywords={}
}

\usepackage{algpseudocode}
\usepackage{bookmark}

\newcommand{\erase}[1]{}

\theoremstyle{plain}
\newtheorem{theorem}{Theorem}[section]
\newtheorem{lemma}[theorem]{Lemma}
\newtheorem{proposition}[theorem]{Proposition}
\newtheorem{corollary}[theorem]{Corollary}

\theoremstyle{definition}
\newtheorem{definition}[theorem]{Definition}
\newtheorem{example}[theorem]{Example}

\theoremstyle{remark}
\newtheorem{remark}[theorem]{Remark}

\numberwithin{equation}{section}
\numberwithin{table}{section}
\numberwithin{figure}{section}
\erase{
\renewcommand{\qed}{\hfill {$\Box$}}
}

\renewcommand{\AA}{\mathord{\mathbb A}}

\newcommand{\CC}{\mathord{\mathbb C}}

\newcommand{\PP}{\mathord{\mathbb  P}}

\newcommand{\RR}{\mathord{\mathbb R}}

\newcommand{\ZZ}{\mathord{\mathbb Z}}

\newcommand{\AAA}{\mathord{\mathcal A}}

\newcommand{\LLL}{\mathord{\mathcal L}}

\newcommand{\NNN}{\mathord{\mathcal N}}

\newcommand{\PPP}{\mathord{\mathcal P}}

\newcommand{\UUU}{\mathord{\mathcal U}}

\newcommand{\maprightsp}[1]{\; \smash{\mathop{\; \longrightarrow \; }\limits\sp{#1}}\; }
\newcommand{\maprightsb}[1]{\; \smash{\mathop{\; \longrightarrow \; }\limits\sb{#1}}\; }

\newcommand{\mapdown}{\phantom{\Big\downarrow}\hskip -8pt \downarrow}
\newcommand{\mapdownright}[1]{\mapdown\rlap{$\vcenter{\hbox{$\scriptstyle#1$}}$}}
\newcommand{\mapdownleft}[1]{\llap{$\vcenter{\hbox{$\scriptstyle#1$\;}}$}\mapdown}

\newcommand{\inj}{\hookrightarrow}

\newcommand{\isom}{\xrightarrow{\sim}}

\newcommand{\set}[2]{\{\,{#1}\mid {#2} \,\}}

\newcommand{\tensor}{\otimes}

\newcommand{\inv}{\sp{-1}}
\newcommand{\dual}{\sp{\vee}}

\newcommand{\sprime}{\sp{\prime}}

\newcommand{\spprime}{\sp{\prime\prime}}
\newcommand{\spcirc}{\sp{\mathord{\circ}}}

\newcommand{\sperp}{\sp{\perp}}

\newcommand{\bdr}{\partial\,}

\DeclareMathOperator{\Image}{Im}
\DeclareMathOperator{\Sing}{Sing}

\DeclareMathOperator{\Ker}{Ker}

\DeclareMathOperator{\Hom}{Hom}

\DeclareMathOperator{\rank}{rank}
\DeclareMathOperator{\pr}{pr}
\DeclareMathOperator{\disc}{disc}

\newcommand{\id}{\mathord{\mathrm{id}}}

\newcommand{\rmand}{\textrm{and}}
\newcommand{\rmor}{\textrm{or}}

\newcommand{\quand}{\quad\rmand\quad}

\newcommand{\intf}[1]{\langle #1 \rangle}

\newcommand{\Cham}{\boldsymbol{C\hskip -.3pt h}}
\newcommand{\Chamb}{\Cham_{\mathrm{b}}}

\newcommand{\Arr}{\mathcal{A}}
\newcommand{\ArrC}{\Arr_{\CC}}
\newcommand{\SingBC}{\PPP}

\newcommand{\thei}{\sqrt{-1}\,}

\newcommand{\res}{|}

\newcommand{\AtR}{\AA^2(\RR)}
\newcommand{\AtC}{\AA^2(\CC)}
\newcommand{\BR}{B(\RR)}
\newcommand{\BC}{B(\CC)}

\newcommand{\Vertexes}{\mathord{\mathrm{Vert}}}
\newcommand{\spbcirc}{\sp{\bullet}}
\newcommand{\betash}{\beta^{\sharp}}

\newcommand{\sigmaA}{\sigma_{\AA}}
\newcommand{\betaRR}{\beta_{\RR}}

\newcommand{\varee}{\varepsilon}
\newcommand{\vareep}{{\varepsilon\sprime}}
\newcommand{\prR}{\pr_{\RR}}

\newcommand{\posR}{\RR_{>0}}
\newcommand{\negR}{\RR_{<0}}
\newcommand{\nonnegR}{\RR_{\ge 0}}

\newcommand{\tilxi}{\tilde{\xi}}
\newcommand{\tilx}{\tilde{x}}
\newcommand{\tilu}{\tilde{u}}
\newcommand{\tilQ}{\widetilde{Q}}

\DeclareMathOperator{\Real}{Re}
\DeclareMathOperator{\Imag}{Im}

\newcommand{\rP}{\mathord{\mathrm P}}
\newcommand{\iP}{i\mathord{\mathrm P}}

\newcommand{\ori}{\mathrm{ori}}
\renewcommand{\ss}{\sigma}
\newcommand{\Pp}{P\sprime}

\newcommand{\spsh}{\sp{\sharp}}
\newcommand{\stdvan}{\Sigma}
\newcommand{\pants}{\Delta}

\newcommand{\intval}{I_{\varee}}
\newcommand{\intvalp}{I_{\vareep}}

\newcommand{\YC}{Y(\CC)}
\newcommand{\YR}{Y(\RR)}
\newcommand{\XC}{X}

\newcommand{\Cp}{C\sprime}

\newcommand{\projNN}{p_{\NNN}}

\newcommand{\Tl}[1]{T[#1]}

\newcommand{\Der}[1]{\partial /\partial #1}
\newcommand{\fDDer}[2]{\frac{\partial #1}{\partial #2}}
\newcommand{\sppprime}{\sp{\prime\prime\prime}}
\newcommand{\tilR}{\widetilde{R}}

\newcommand{\spm}{^{-}}
\newcommand{\ellej}{\ell_{ej}}
\newcommand{\ellPj}{\ell_{Pj}}
\newcommand{\ellPpj}{\ell_{\Pp j}}
\newcommand{\ellez}{\ell_{e0}}
\newcommand{\ellPz}{\ell_{P0}}
\newcommand{\ellPpz}{\ell_{\Pp 0}}
\newcommand{\gp}{g\sprime}

\newcommand{\UW}[1]{U_{W, #1}}
\newcommand{\UX}[1]{U_{X, #1}}

\newcommand{\CCCC}{\mathord{\mathscr{C}}}
\newcommand{\CCCCb}{\mathord{\boldsymbol{C}}}

\newcommand{\tilell}{\tilde{\ell}}
\newcommand{\tilrho}{\tilde{\rho}}
\newcommand{\tilpi}{\tilde{\pi}}
\newcommand{\tilX}{\widetilde{X}}
\newcommand{\tilW}{\widetilde{W}}
\newcommand{\tilArrC}{\widetilde{\Arr}_{\CC}}
\newcommand{\tilBC}{\widetilde{B}(\CC)}

\newcommand{\PtC}{\PP^2(\CC)}
\newcommand{\intfnull}{\intf{\phantom{aa}}}

\newcommand{\ellinf}{\tilell_{\infty}(\CC)}
\newcommand{\PPPinf}{\PPP_{\infty}}
\newcommand{\aone}{\mathord{\mathrm{a1}}}
\newcommand{\dfour}{\mathord{\mathrm{d4}}}
\newcommand{\PPPinfa}{\PPP_{\infty, \aone}}
\newcommand{\PPPinfd}{\PPP_{\infty, \dfour}}
\newcommand{\Hinf}{H_{\infty}}
\newcommand{\Rinf}{R_{\infty}}
\newcommand{\Rinfa}{R_{\infty, +}}
\newcommand{\Rinfb}{R_{\infty, -}}
\newcommand{\LLLinf}{\LLL_{\infty}}
\newcommand{\rinf}{r_{\infty}}

\newcommand{\varGammaPtriangle}{\varGamma_P}

\newcommand{\nopslines}{N}
\newcommand{\discg}{\mathord{\mathrm{disc}}}
\newcommand{\barH}{\overline{H}}
\newcommand{\tilN}{\widetilde{\nopslines}}

\newcommand{\vect}[1]{\boldsymbol{#1}}
\newcommand{\free}{\mathrm{fr}}

\begin{document}

\title[Real line arrangement]
{Topology of a complex double plane branching along a real line arrangement}

\author[Ichiro  Shimada]{Ichiro Shimada}
\address{Mathematics Program \\
Graduate School of Advanced Science and Engineering \\
Hiroshima University \\
1-3-1 Kagamiyama, Higashi-Hiroshima, 739-8526 JAPAN}
\email{ichiro-shimada@hiroshima-u.ac.jp}

\begin{abstract}
We investigate the topology of the double cover of the complex affine plane 
branching along a nodal real line arrangement.
We define certain topological $2$-cycles in the double plane 
using the real structure of the arrangement,
and calculate their intersection numbers.
\end{abstract}


\thanks{This work was supported by JSPS KAKENHI,  Grant Number~20K20879,~20H01798,~23H00081 and~20H00112.} 
\thanks{The author declares that there is no conflict of interest. The data supporting the findings of this study are available from the author upon reasonable request.}
\maketitle
\section{Introduction}\label{sec:introduction}
Explicit description of topological cycles in complex algebraic varieties is an important task.
It is essential, for example,
in the study of periods and related differential equations in a family of  algebraic varieties.
For the recent development of 
\emph{numerical algebraic geometry}~(see, for example,~\cite{ElsenhansJorg2024, LairezSertoz2019}),
we need explicit descriptions of topological cycles suited for the numerical computation of periods
by multiple integrals.
For example, in~\cite{CynkVanStraten2019}, 
the arithmetic of certain Calabi-Yau threefolds obtained as 
the minimal desingularizations of 
double covers of $\PP^3$ is studied by  numerical integration over some topological $3$-cycles.
\par
The first major general theory on topological cycles in 
complex algebraic varieties is the theory of vanishing cycles 
due to Lefschetz~\cite{Lefschetz1924}.
(See also~\cite{Lamotke1981} for the modern accounts of this theory.)
In the present work,
we investigate the topology of the smooth algebraic surface $X$
obtained as the minimal desingularization of 
the double cover of the complex affine plane
branching along a nodal \emph{real} line arrangement.
Using the real structure of the arrangement, 
we construct certain topological $2$-cycles in $X$,
which resemble, in some way, the vanishing cycles of Lefschetz.
\par
To understand the topology of an algebraic surface,
it is important to calculate the intersection form on the middle homology group.
The main purpose of this paper is to calculate the intersection numbers of our topological $2$-cycles in $X$.
Since some pairs of these cycles intersect in  loci of real dimension $\ge 1$,
we have to construct small displacements of these cycles.
\par
Let $\AtR$ be a real affine plane.
An arrangement of ${\nopslines}$ real lines 
 \[
\Arr:=\{\ell_1(\RR ), \dots, \ell_{\nopslines}(\RR )\}
\]
on $\AtR$ is called \emph{nodal} 
if  no three lines of $\Arr$ are concurrent.
Suppose that $\Arr$ is 
a nodal real line arrangement.
Let  
\[
\ArrC :=\{\ell_1(\CC ), \dots, \ell_{\nopslines}(\CC )\}
\]
be the arrangement of complex affine lines in the complex affine plane  $\AtC$
obtained by complexifying the lines in $\Arr$.
We put
\[
\BR:=\bigcup_{i=1}^{\nopslines} \ell_i(\RR ),
\quad 
\BC:=\bigcup_{i=1}^{\nopslines} \ell_i(\CC ).
\]
We consider the morphisms 
\begin{equation}\label{eq:rhopi}
X\maprightsp{\rho} W \maprightsp{\pi} \AtC,
\end{equation}
where  $\pi\colon W\to \AtC$ is the double covering whose branch locus is equal to $\BC$,
and $\rho\colon X\to W$ is the minimal resolution of singularities.
The purpose of this paper is to investigate  the intersection form
\[
\intfnull\;\;\colon\;\;  H_2(X; \ZZ)\times H_2(X; \ZZ)\to  \ZZ
\]
on the second homology group $H_2(X; \ZZ)$ of $X$,
where $X$ is oriented as a smooth complex surface.
\par
For simplicity, we put 
\[
\SingBC:=\Sing \BC.
\]
Let $\beta\colon \YC\to \AtC$ be the blowing-up at the points in $\SingBC$.
For a subset $S$ of $\AtC$, we put
\[
S\spbcirc:=S\setminus(S\cap \SingBC),
\quad
\betash S:=\textrm{the closure of $\beta\inv(S\spbcirc)$ in $\YC$},
\]
and call $\betash S$ the \emph{strict transform} of $S$.
We then put
\[
\YR:=\betash \AtR. 
\]
Note that $\betash B(\CC)$ is the disjoint union of smooth rational curves $\betash\ell_i(\CC )$ on $\YC$.
Then $X$ fits in the following commutative diagram:
\begin{equation}\label{eq:commdiagXWYA}
\begin{array}{ccc}
X  &\maprightsp{\rho} & W \\
\mapdownleft{\phi}  && \mapdownright{\pi}\\
\YC & \maprightsb{\beta} & \AtC,
\end{array}
\end{equation}
where $\phi\colon X\to \YC$ is the double covering whose branch locus is 
equal to $\betash B(\CC)$.
For $P\in \SingBC$, 
let $E_P$ denote the exceptional $(-1)$-curve of $\beta$ over $P$,
and we put 
\[
D_P:=\phi\inv(E_P)=(\pi\circ \rho)\inv (P),
\]
which is a smooth rational curve on $X$ with self-intersection number $-2$.
Note  that $D_P$ is the exceptional curve of  the minimal desingularization 
$\rho\colon X\to W$ over the ordinary node of $W$ that is mapped to $P$ by $\pi$.
If $P$ is the intersection point of the lines $\ell_a(\RR )$ and $\ell_b(\RR )$ in $\Arr$,
then  $\phi\res D_P\colon D_P\to E_P$
is the double covering branching at 
the intersection points 
of $E_P$ and
the strict transforms $\betash\ell_a(\CC )$ and $\betash\ell_b(\CC )$.
\par
A \emph{chamber} is the closure in $\AtR$ of a connected component  of  
 the complement $\AtR \setminus \BR$ of $\BR$.
We denote by  $\Chamb$  the set of \emph{bounded} chambers.
 Let $C$ be a bounded chamber.
We  put  
\[
\Vertexes (C):=C\cap \SingBC.
\]
A point of $\Vertexes (C)$ is called a \emph{vertex} of $C$.
Let $P$ be a vertex of $C$.
Then $\YR\cap E_P$ is a circle on the Riemann sphere $E_P$, 
on which the two branch points of the double covering $\phi|D_P\colon D_P\to E_P$ locate,
and 
\[
J_{C, P}:=\betash C\cap E_P
\] 
is a part of the circle  $\YR\cap E_P$ connecting
these two branch points.
Therefore 
\[
S_{C, P}:=\phi\inv(\betash C)\cap D_P=\phi\inv (J_{C, P})
\]
is a circle on the Riemann sphere $D_P$.
The space $\phi\inv(\betash C)$ is homeomorphic to a $2$-sphere
minus a union of  disjoint open discs,
and we have
\[
\bdr \phi\inv(\betash C)=\bigsqcup_{P\in\Vertexes (C)}  S_{C, P}.
\]
Let $\gamma_C$ be an orientation  of $\phi\inv(\betash C)$.
We denote by $\pants(C, \gamma_C)$
the topological $2$-chain $\phi\inv(\betash C)$ oriented by $\gamma_C$.
Each boundary component $S_{C, P}$ of $\pants(C, \gamma_C)$
is oriented by $\gamma_{C}$.
Note that  $S_{C, P}$ divides the $2$-sphere $D_P$ into the union of two closed hemispheres, 
and that the two  hemispheres with their complex structures induce orientations on $S_{C, P}$ that
are opposite to each other.
\begin{definition}\label{def:capping}
The \emph{capping hemisphere} for  $\gamma_C$ at $P$ is 
the closed hemisphere $H_{C, \gamma_C, P}$ on 
$D_P$ with  $\bdr H_{C, \gamma_C, P}=S_{C, P}$
such that 
the orientation on $S_{C, P}$ induced by the complex structure of $H_{C, \gamma_C, P}$
is opposite to the orientation  
induced by the orientation $\gamma_C$ of 
$\pants(C, \gamma_C)$.
\end{definition}
Let $H_{C, \gamma_C,  P}$ be the capping hemisphere  for $\gamma_C$ at $P$.
Then
\[
\stdvan(C, \gamma_C):=\pants(C, \gamma_C) \;\;\cup\;\; \bigsqcup_{P\in\Vertexes (C)}  H_{C, \gamma_C, P}
\]
with the orientations coming from the complex structure on each $H_{C, \gamma_C, P}$
is a topological $2$-cycle homeomorphic to a $2$-sphere.
Figure~\ref{fig:vanishingcycle} illustrates $\stdvan(C, \gamma_C)$
when $C$ is a triangle.
\par
In  this paper, we show that this topological $2$-cycle $\stdvan(C, \gamma_C)$ 
resembles, in many aspects,
the vanishing cycle for an ordinary node of a complex surface.
Hence we make the  following:
\begin{definition}\label{def:vanishing}
The $2$-cycle $\stdvan(C, \gamma_C)$ is called a \emph{vanishing cycle} 
over the bounded chamber $C$.
Its homology class $[\stdvan(C, \gamma_C)]\in H_2(X; \ZZ)$ is also called 
a vanishing cycle over $C$.
\end{definition}
\begin{figure}
\begin{tikzpicture}[x=2cm, y=2cm]
%
\draw[thick](-0.75,0.8660254037844386)--(-0.75,1.166025403784439);
\draw[thick](0.75,0.8660254037844386)--(0.75,1.166025403784439);
\draw[thick](-0.3761491871287599,-1.082132981209968)--(-0.6361159271456396,-1.231857041016654);
\draw[thick](-1.12476948616219,0.2177007188744306)--(-1.384736226179069,0.06797665906774464);
\draw[thick](1.125457222599352,0.2141168841986855)--(1.384945727466262,0.06356551143199263);
\draw[thick](0.3727003587658874,-1.083325640135864)--(0.6321888636327971,-1.233877012902557);
\draw[thick] (-1.12476948616219,0.2177007188744306) arc [start angle=-60, end angle=0, radius=0.75];
\draw[thick] (0.3727003587658874,-1.083325640135864) arc [start angle=60, end angle=120, radius=0.75];
\draw[thick] (0.75,0.8660254037844386) arc [start angle=180, end angle=240, radius=0.75];
\draw[thick] (0.75,1.166025403784439) arc [start angle=0, end angle=180, radius=0.75];
\draw[thick] (-1.384736226179069,0.06797665906774464) arc [start angle=120, end angle=300, radius=0.75];
\draw[thick] (0.6321888636327971,-1.233877012902557) arc [start angle=240, end angle=420, radius=0.75];
\draw[thick, rotate=0,shift={(0,1.166025403784439)}, variable =\t, domain =0:180]plot[smooth]({0.75*cos(\t)}, {0.2*sin(\t)});
\draw[thick, rotate=120,shift={(0,1.166025403784439)}, variable =\t, domain =0:180]plot[smooth]({0.75*cos(\t)}, {0.2*sin(\t)});
\draw[thick, rotate=240,shift={(0,1.166025403784439)}, variable =\t, domain =0:180]plot[smooth]({0.75*cos(\t)}, {0.2*sin(\t)});
\draw[thin, rotate=0,shift={(0,1.166025403784439)}, variable =\t, domain =180:360]plot[smooth]({0.75*cos(\t)}, {0.2*sin(\t)});
\draw[thin, rotate=120,shift={(0,1.166025403784439)}, variable =\t, domain =180:360]plot[smooth]({0.75*cos(\t)}, {0.2*sin(\t)});
\draw[thin, rotate=240,shift={(0,1.166025403784439)}, variable =\t, domain =180:360]plot[smooth]({0.75*cos(\t)}, {0.2*sin(\t)});
\node [right] at (-0.4,0){$\pants(C, \gamma_C)$};
\node [right] at (0.75,1.166025403784439){$S_{C,P}$};
\node [right] at (-0.3,1.646025403784439){$H_{C,\gamma_C, P}$};
%
\end{tikzpicture}
\caption{Vanishing cycle $\stdvan(C, \gamma_C)$}\label{fig:vanishingcycle}
\end{figure}
By definition,
the homology class $[\stdvan(C, \gamma_C)]$
 depends on 
 $ \gamma_C$ as follows:
\begin{equation}\label{eq:orientationreversing}
[\stdvan(C, \gamma_C)]+[\stdvan(C, -\gamma_C)]=\sum_{P\in \Vertexes  (C)} [D_P], 
\end{equation}
where $[D_P] \in H_2(X; \ZZ)$ is the homology class of the smooth rational curve $D_P$.
\par
Our first main result is as follows.
\begin{theorem}\label{thm:basis}
The $\ZZ$-module $H_2(X; \ZZ)$ is free.
We fix an orientation $\gamma_C$
for each $C\in \Chamb$.
Then the homology classes $[\stdvan(C, \gamma_C)]$,
where $C$ runs through $\Chamb$, 
and the homology classes $[D_P]$,
where $P$ runs through $\SingBC$, 
form a basis of $H_2(X; \ZZ)$.
\end{theorem}
To state the second result, we need the following:
\begin{definition}\label{def:coherent}
Let $C$ and $C\sprime$ be bounded chambers such that 
$C\cap C\sprime$ consists of a single point $P\in \SingBC$.
Note that we have $S_{C, P}=S_{C\sprime, P}$.
We say that the orientations $\gamma_C$ and $\gamma_{C\sprime}$  are \emph{coherent}
if the capping hemispheres  $H_{C, \gamma_C, P}$ and  $H_{C\sprime, \gamma_{C\sprime}, P}$ 
are distinct.
We say that a collection $\set{\gamma_C}{C\in \Chamb}$ of orientations 
are \emph{coherent} if 
 $\gamma_C$ and $\gamma_{C\sprime}$ are coherent
whenever 
$C\cap C\sprime$ consists of a single point.
 \end{definition}
We have the following:
\begin{proposition}\label{prop:coherentcollection}
A coherent collection  of orientations exists.
\end{proposition}
\begin{definition}\label{def:edge}
Let $C$ be a chamber.
We say that 
 $\ell_i(\RR)\in \Arr$ defines an \emph{edge} $C\cap \ell_i(\RR)$ of  $C$ if
 $C\cap \ell_i(\RR)$ contains a non-empty open subset of  $\ell_i(\RR)$.
 \end{definition}
Our second main result is as follows.
\begin{theorem}\label{thm:intnumbs}
For each bounded chamber $C\in \Chamb$,
we fix an orientation $\gamma_C$.
Let $C$ and $C\sprime$ be distinct bounded chambers.
\begin{enumerate}[{\rm (1)}]
\item For $P\in \SingBC$, we have
\[
\intf{[\stdvan(C, \gamma_C)], [D_P]}=\begin{cases}
-1 & \textrm{if $P\in \Vertexes (C)$}, \\
0 & \textrm{if $P\notin \Vertexes (C)$}.
\end{cases}
\]
\item
If $C$ and $C\sprime$ are disjoint,  then $\intf{[\stdvan(C, \gamma_C)], [\stdvan(C\sprime, \gamma_{C\sprime})]}=0$.
\item
Suppose that $C\cap C\sprime$ consists of a single point.
Then we have
\[
\intf{[\stdvan(C,\gamma_C)], [\stdvan(C\sprime, \gamma_{C\sprime})]}=
\begin{cases}
0 &\textrm{if $\gamma_C$ and $\gamma_{C\sprime}$ are coherent}, \\
-1 &\textrm{otherwise}.
\end{cases}
\]
\item
If $C$ and $C\sprime$ share a common edge,  
then 
$\intf{[\stdvan(C,\gamma_C)],  [\stdvan(C\sprime, \gamma_{C\sprime})]}=-1$.
\item The self-intersection number $\intf{[\stdvan(C, \gamma_C)], [\stdvan(C, \gamma_C)]}$ is equal to $-2$.
\end{enumerate}
\end{theorem}
\begin{remark}\label{rem:compatibility}
Using~\eqref{eq:orientationreversing},
we can easily check that, if 
Theorem~\ref{thm:intnumbs} is valid for one choice of  
orientations, 
then Theorem~\ref{thm:intnumbs} is true for any choice of
orientations.
Hence it is enough to prove Theorem~\ref{thm:intnumbs}
for a fixed collection of orientations.
We will prove Theorem~\ref{thm:intnumbs}
for  \emph{standard orientations},
which will be  defined in Section~\ref{subsec:StandardOrientations}.
The collection of standard orientations is in fact coherent.
\end{remark}
Our construction of vanishing cycles over bounded chambers
is similar to the construction of topological cycles by Pham~\cite{Pham1965}
in hypersurfaces defined  by equations of Fermat type.
Pham's construction has many applications.
For example, it was applied in~\cite{DegtyarevShimada2016}
to the study of integral Hodge conjecture for Fermat varieties.
It was also applied to the calculation of 
Picard lattices of quartic surfaces in~\cite{LairezSertoz2019}.
\par
A totally different algorithm to calculate 
the intersection form on the middle homology group of an open complex surface,
which we called a \emph{Zariski-van Kampen method}, 
was presented in~\cite{ArimaShimada2009} and~\cite{ShimadaTakahashi2010}.
It was applied to the construction of arithmetic Zariski pairs
in~\cite{ArimaShimada2009} and~\cite{Shimada2008}.
\par
For other results on double coverings of the plane and real algebraic geometry, 
the reader is referred to~\cite{Shustin1990}.
\par
This paper is organized as follows.
In Section~\ref{sec:Terminology},
we fix terminology about real line arrangements.
In particular,
we define standard orientations $\sigma_C$.
In Section~\ref{sec:proofThmA},
we prove Theorem~\ref{thm:basis}.
In Section~\ref{sec:cappings},
we calculate the capping hemispheres explicitly.
In Section~\ref{sec:displacements},
we introduce   the notion of   displacements of vanishing  cycles.
In Sections~\ref{sec:proofofThmB},~\ref{sec:Proof4},~\ref{sec:Proof5}, 
we prove Theorem~\ref{thm:intnumbs}.
 In Section~\ref{sec:examples},
 we calculate some concrete examples.
 \begin{remark}
 A referee for the first version of this paper 
 suggests a simpler proof of Theorem~\ref{thm:intnumbs} (5),
 which uses a normal vector field.
 \end{remark}
\section{Terminology}\label{sec:Terminology}
Let $\AAA=\{\ell_1(\RR ), \dots, \ell_{\nopslines}(\RR )\}$ be a nodal real line arrangement.
\subsection{Standard orientations}\label{subsec:StandardOrientations}
We define \emph{standard orientations}
$\sigma_C$ for bounded chambers $C$.
\begin{definition}
A \emph{defining polynomial} of the arrangement $\AAA$  is the product 
\[
f:=\prod_{i=1}^n  \lambda_i,
\]
where $\lambda_i$ is  an affine linear function 
$\lambda_i\colon \AtR\to \RR $
such that $\ell_i(\RR )=\lambda_i\inv (0)$.
\end{definition}
A defining polynomial of $\AAA$ is unique up to real multiplicative constant.
We can regard a defining polynomial  $f\colon \AtR\to \RR $ of $\AAA$ as a complex-valued polynomial function $f\colon \AtC\to \CC$.
\par
We  fix an orientation $\sigmaA$ of   $\AtR$
and  a  defining polynomial $f$ of $\AAA$.
Then the double covering $\pi\colon W\to \AtC$ in the diagram~\eqref{eq:commdiagXWYA}
is given by the projection to the second factor 
 from 
\begin{equation}\label{eq:Weq}
W:=\set{(\omega, Q)\in \CC \times \AtC}{\omega^2=f(Q)}, 
\end{equation}
where $\omega$ is an affine coordinate of $\CC $.
Let $C$ be  a bounded chamber of $\Arr$, and 
let $C\spcirc$ be the interior of $C$ in $\AtR$.
The pull-back  $\pi\inv(C\spcirc)$ has two connected components,(
which we call \emph{sheets}.
We can consider  $\pi\inv(C\spcirc)$ as an open subset  of 
$\phi\inv(\betash C)$   via 
 $\rho\colon X\to W$,
 where $\phi$ and $\rho$ are given~\eqref{eq:commdiagXWYA}.
Let $\gamma_C$ be an orientation of $\phi\inv(\betash C)$.
We let the two sheets be oriented by 
$\gamma_C$.
Then $\pi\colon W\to\AtC$ restricted to one  sheet 
is an orientation-preserving homeomorphism to $C\spcirc$  oriented by 
$\sigmaA$,
whereas 
$\pi$ restricted to the other  sheet is orientation-reversing.
\begin{definition}\label{def:sheets}
The sheet on which $\pi$ is orientation-preserving
is called the \emph{positive-sheet} with respect to $\sigmaA$ and $\gamma_C$,
and the other sheet is called the \emph{negative-sheet}.
\end{definition}
The orientation $\gamma_C$ of $\phi\inv(\betash C)$ 
 is specified by indicating which sheet  is positive.
Note that we have either $f(C\spcirc)\subset \RR_{> 0}$ or $f(C\spcirc)\subset \RR_{< 0}$.
In the former case,
the two sheets  are distinguished by the sign on  $\phi\inv(\betash C)$ of the function $\omega=\pm \sqrt{f}\in \RR $,
and in the latter case,
the two  are distinguished by the sign of  $\omega/\thei=\pm\sqrt{-f}\in \RR $.
\begin{definition}\label{def:sigmaC}
The \emph{standard orientation} $\sigma_C$ 
for a bounded chamber $C$ with respect to $\sigmaA$ and $f$ is 
the orientation of $\phi\inv(\betash C)$ such that 
\begin{itemize}
\item
when  $f(C\spcirc)\subset \RR_{>0}$,
 the sheet with $\omega>0$ is positive,  and 
\item
when  $f(C\spcirc)\subset \RR_{<0}$,
 the sheet with $\omega/\thei>0$ is positive.
\end{itemize}
\end{definition}
From now on to the end of this paper, 
we fix  $f$ and $\sigmaA$, and use the standard orientation $\sigma_C$.
We omit $\sigma_C$
from the notation $H_{C, \sigma_C, P}$, $\pants(C,  \sigma_C)$, and $\stdvan(C, \sigma_C)$
introduced in Section~\ref{sec:introduction}:
\begin{equation}\label{eq:omit}
H_{C,  P}:=H_{C, \sigma_C, P},
\quad 
\pants(C):=\pants(C,  \sigma_C), 
\quad
\stdvan(C):=\stdvan(C, \sigma_C).
\end{equation} 
In Section~\ref{subsec:capping}, 
we show that 
the collection $\set{\sigma_C}{C\in \Chamb}$ of standard orientations is coherent in the sense of Definition~\ref{def:coherent}.
In Sections~\ref{sec:proofofThmB}, ~\ref{sec:Proof4},~\ref{sec:Proof5},
we prove Theorem~\ref{thm:intnumbs}
for the standard orientations. 
See~Remark~\ref{rem:compatibility}.
\subsection{The vector space of translations}\label{subsec:translations}
Let $T(\AtC)$ denote the $\CC$-vector space of translations of $\AtC$,
and let $T(\AtR)$ denote  the $\RR$-vector space of translations of $\AtR$.
For a vector $\tau \in T(\AtC)$ and $P\in \AtC$, we denote by  $P+\tau$ the image of $P$ by the translation $\tau\colon \AtC\to\AtC$.
For $Q, Q\sprime\in \AtC$, let $\tau_{Q, Q\sprime}\in T(\AtC)$ 
denote the unique translation such that $Q+\tau_{Q, Q\sprime}=Q\sprime$.
We have
\[
T(\AtC)=T(\AtR)\tensor\CC=T(\AtR)\oplus \thei T(\AtR).
\]
Let $\overline{Q}$ denote the complex conjugate of the point $Q\in \AtC$.
Then the mapping 
\[
Q\;\mapsto\; Q+(1/2) \tau_{Q, \overline{Q}}\;=\;(Q+\overline{Q})/2
\]
 that gives the real part of $Q\in \AtC$
yields a projection 
\[
\prR \colon \AtC\to \AtR.
\]
Then we have a natural identification 
\begin{equation}\label{eq:naturalidentification}
\prR\inv (Q)\;\cong\; \thei T(\AtR)
\end{equation}
of the fiber $\prR\inv (Q)$ 
over $Q\in \AtR$ with  $\thei T(\AtR)$ 
by
$Q\sprime\mapsto  \tau_{Q, Q\sprime}$.
For a real affine line $\lambda (\RR)\subset \AtR$,
we put
\[
\Tl{\lambda(\RR)}:=\set{\tau\in T(\AtR)}{\lambda(\RR)+\tau=\lambda(\RR)},
\]
which is a $1$-dimensional $\RR$-vector subspace of $T(\AtR)$.
In the same way,
we define the subspace of $\Tl{\lambda(\CC)}\subset T(\AtC)$ 
for a complex affine line $\lambda (\CC)\subset \AtC$.
If $\lambda (\CC)$ is 
 the complexification $\lambda(\RR)\tensor\CC$ of a real affine line $\lambda(\RR)$,
then we have
\[
\Tl{\lambda(\RR)\tensor\CC}=\Tl{\lambda(\RR)}\tensor\CC=\Tl{\lambda(\RR)}\oplus \thei \Tl{\lambda(\RR)}.
\]
For affine coordinates  $(x, y)$ of $\AtR$,
we define a basis $e_x, e_y$ of $T(\AtR)$ by
\begin{equation}\label{eq:exey}
e_x\colon (x, y)\mapsto (x+1, y),
\quad 
e_y\colon (x, y)\mapsto (x, y+1).
\end{equation}
Suppose that a real line $\lambda(\RR)\subset \AtR$ is 
defined by 
\[
ax+by+c=0,
\]
where $a, b, c\in \RR$.
Then the linear subspace $T[\lambda(\RR)]$ 
of $T(\AtR)$ (and hence  the linear subspace $T[\lambda(\RR)\tensor\CC]$ 
of $T(\AtC)$)
is generated by the vector
\[
b e_x-a e_y.
\]

\subsection{Good coordinates}\label{subsec:goodcoordinates} 
Let  $(\xi, \eta)$ be affine coordinates   of $\AtC$.
We put
\[
\xi=x+\thei u,
\quad
\eta=y+\thei v,
\]
where $x, u, y, v$ are real-valued functions on $\AtC$.
We say that $(\xi, \eta)$ is \emph{compatible with the $\RR$-structure} if 
the complex conjugation  $Q\mapsto \overline{Q}$ is given by 
\[
(x+\thei u, y+\thei v) \;\mapsto\; (x-\thei u, y-\thei v).
\]
Suppose that $(\xi, \eta)$ is compatible with the $\RR$-structure.
Then 
$\AtR$ is equal to $\{u=v=0\}$, 
and 
the restriction of $(x, y)$ to $\AtR$  
is an affine coordinate system of $\AtR$.
From now on, 
we regard $(x, y)$ as affine coordinates of $\AtR$
by restriction.
Then $\prR\colon \AtC\to \AtR$ is given by
\[
(\xi, \eta) \;\mapsto\; (x, y).
\]
We say that 
 the affine coordinates $(\xi, \eta)$ compatible with the $\RR$-structure are  \emph{good coordinates}
if the ordered basis $(\partial/\partial x, \partial/\partial y)$ 
of tangent vectors of $\AtR$ is positive with respect to the orientation $\sigmaA$ of $\AtR$
fixed in Section~\ref{subsec:StandardOrientations}.
\section{Proof of Theorem~\ref{thm:basis}}\label{sec:proofThmA}
In this section,
we prove Theorem~\ref{thm:basis}.
For a topological space $T$,
we write $H_2(T)$ for $H_2(T; \ZZ)$.
Recall the commutative diagram~\eqref{eq:commdiagXWYA}.
\begin{lemma}\label{lem:kernel}
The homomorphism $\rho_{*} \colon H_2(X)\to H_2(W)$ is surjective,
and its kernel is a 
free $ \ZZ$-module with basis $[D_P]$, where $P$ runs through $\SingBC=\Sing B(\CC)$.
\end{lemma}
\begin{proof}
This follows from Lemma~4.4~of~\cite{ElsenhansJorg2024}.
\end{proof}
Next we prove the following:
\begin{lemma}\label{lem:image}
The homology group  $H_2(W)$ is a 
free $ \ZZ$-module of which the classes   $\rho_*([\stdvan(C)])$ form a basis, where $C$ runs through $\Chamb$.
\end{lemma}
Combining Lemmas~\ref{lem:kernel} and~\ref{lem:image},
we will obtain a proof of Theorem~\ref{thm:basis}.
\par
Remark that Lemma~\ref{lem:image} holds trivially when $\Chamb=\emptyset$.
Indeed, if $\Chamb$ is empty, 
then either all lines in $\Arr$ are parallel, 
or all lines in $\Arr$ except one line are parallel.
In these cases,  we have $H_2(W)=0$, 
because there exists a deformation retraction of $W$ onto a $\mathrm{CW}$-complex of dimension $1$. 
\par
From now on, we assume $\Chamb\ne \emptyset$.
We  put
\[
W_{\RR }:=\pi \inv (\AtR), \;\;
\CCCC:=\bigcup_{C\in \Chamb} C,  \;\;
W_{\CCCC}:=\pi \inv(\CCCC)=\bigcup_{C\in \Chamb} \pi\inv (C). 
\]
Note that, for $C\in \Chamb$,
the pull-back $\pi\inv(C)$ is homeomorphic to a $2$-sphere,
and is equal to the image $\rho(\stdvan(C))$
of the vanishing cycle $\stdvan(C)$ by the minimal desingularization $\rho\colon X\to W$.
Hence Lemma~\ref{lem:image} follows from the following two lemmas:
\begin{lemma}\label{lem:retract}
There exists a strong deformation retraction of $W $ onto  
$W_{\CCCC}$.
\end{lemma}
\begin{lemma}\label{lem:bouquet}
The homology group $H_2(W_{\CCCC})$ is a free $\ZZ$-module with basis
$[\pi\inv(C)]$, where $C$ runs through $\Chamb$.
Here we understand that 
each $2$-sphere $\pi\inv(C)$ is equipped with an orientation.
\end{lemma}
\begin{proof}[Proof of Lemma~\ref{lem:retract}]
First, we construct a strong deformation retraction of $W $ onto  
$W_{\RR }$.
To begin with,  
we construct a strong deformation retraction
\[
F\colon \AtC\times I\to \AtC
\]
of $\AtC$ onto  $\AtR$ such that, for $Q\in \AtC$, we have 
\begin{equation}\label{eq:rinvP}
\textrm{$F(Q, t_0)\in \BC$ for  $t_0\in I$}
\;\;\Longrightarrow\;\;
\textrm{$F(Q, t)\in \BC$ for all $t\in [t_0, 1]$}.
\end{equation}
For $Q_0\in \AtR$,
under the natural identification~\eqref{eq:naturalidentification} of $\prR\inv (Q_0)$ with $\thei T(\AtR)$,
we have
\[
\prR\inv (Q_0)\cap \BC=\begin{cases}
\emptyset & \textrm{if $Q_0\notin \BR$}, \\
\thei T[\ell_i(\RR )] & \textrm{if $Q_0\in \ell_i(\RR) \setminus \SingBC$}, \\
\thei \;(\,T[\ell_i(\RR )] \cup T[\ell_j(\RR )]\, )& \textrm{if $\{Q_0\}=\ell_i(\RR)\cap  \ell_j(\RR)$}.
\end{cases}
\]
Hence, for  $Q\in \AtC$,  the line segment in  $\prR\inv (Q_0)$ connecting $Q$ and its real part 
$Q_0:=\prR(Q)\in \AtR$
is either disjoint from $\BC$ or entirely contained in $\BC$ or intersecting $\BC$ only at $Q_0$.
Therefore the strong deformation retraction
of $\AtC$ onto  $\AtR$ defined by 
\[
F(Q, t):= (1-t)\cdot Q+t \cdot Q_0=Q+t\cdot \tau_{Q, Q_0}
\]
satisfies~\eqref{eq:rinvP}. (Recall that $\tau_{Q, Q_0}\in T(\AtC)$ is the vector $Q_0-Q$.)
Now, since the double covering $\pi\colon W \to \AtC$
is \'etale outside $\BC$, and induces a homeomorphism from $\pi\inv(\BC)$ to $ \BC$, 
the  property~\eqref{eq:rinvP} enables us to construct  a strong deformation retraction 
\[
\widetilde{F}\colon W\times I\to W
\]
of $W $ onto  $W_{\RR }$
as a lift of  $F$.
Indeed, 
for $Q\sprime\in W$, 
we define $\widetilde{F}(Q\sprime, t)$
using the homotopy lifting property of $\pi$ outside $\BC$
as long as $F(\pi(Q\sprime), t)\notin \BC$, 
while if $F(\pi(Q\sprime), t)\in \BC$, 
we define $\widetilde{F}(Q\sprime, t)$ to be the unique point of   $\pi\inv(\BC)$ 
that is mapped to  $ F(\pi(Q\sprime), t)$.
\par
Next we construct a strong deformation retraction of $W_{\RR }$ onto
$W_{\CCCC}$.
Since $\pi\res W_{\RR }\colon W_{\RR }\to \AtR$ is 
a local  isomorphism outside of  $\BR$,
and induces a homeomorphism
from $(\pi\res W_{\RR })\inv(\BR)$ to $\BR$, 
the same argument as above shows that 
it is enough to construct a strong deformation retraction
\[
G\colon \AtR\times I\to \AtR
\]
of $\AtR$ onto  $\CCCC$ that satisfies
\begin{equation*}\label{eq:GBR}
\textrm{$G(Q, t_0)\in \BR$ for  $t_0\in I$}
\;\;\Longrightarrow\;\;
\textrm{$G(Q, t)\in B(\RR)$ for all $t\in [t_0, 1]$}.
\end{equation*}
First we make a strong deformation retraction of each unbounded chamber to its boundary,
so that $ \AtR$ is retracted onto the union  $\CCCC\cup \BR$.
Next we make a strong deformation retraction of $\CCCC\cup \BR$ onto $\CCCC$.
\par
Thus Lemma~\ref{lem:retract} is proved.
\end{proof}
\begin{proof}[Proof of Lemma~\ref{lem:bouquet}]
We put $m:=|\Chamb|$.
It is enough to show that 
there exists a numbering  $C_1, \dots, C_m$ of elements of $\Chamb$
such that, putting
\[
\CCCC_k:=C_1\cup \dots\cup C_k,
\quad 
K_k:=C_k\cap \CCCC_{k-1}\subset \bdr C_k, 
\]
we have $K_k\ne \bdr C_k$ for $k=2, \dots, m$.
Indeed, if such a sequence $C_1, \dots, C_m$ exists,
then we have
$H_1(\pi\inv (K_k))=H_2(\pi\inv (K_k))=0$
for $k=2, \dots, m$, and hence we obtain 
\[
H_2(\pi\inv(\CCCC_k))=H_2(\pi\inv(\CCCC_{k-1}))\oplus H_2(\pi\inv(C_k))
\]
by the Mayer--Vietoris sequence, and  thus 
\[
H_2(W_{\CCCC})=H_2(\pi\inv(\CCCC_m))=H_2(\pi\inv(C_1))\oplus \cdots \oplus H_2(\pi\inv(C_m))
\]
is proved by induction on $k$. 
We construct the \emph{reversed} sequence $C_m, \dots, C_1$
by the following procedure:
\begin{algorithmic} 
\State Set $k:=m$ and $\CCCCb:=\Chamb$
\While{$k>0$}     
\State{Let $\CCCC_k$ be the union of chambers in $\CCCCb$.}
\State{Note that $\CCCC_k$ is bounded.}
\State{We choose $C\in \CCCCb$
such that an edge $e_k$ of $C$ is a part of $\bdr \CCCC_k$.}
\State{We set $C_k:=C$, remove $C$ from $\CCCCb$, and decrement $k$ by $1$.}
 \EndWhile
\end{algorithmic}
Since $K_k=C_k\cap \CCCC_{k-1}$ is contained in $\bdr C_k$ minus the interior of the edge $e_k$, 
 the obtained sequence satisfies
$K_k\ne \bdr C_k$ for $k=2, \dots, m$.
\end{proof}
\section{Capping hemispheres}\label{sec:cappings}
Recall that 
we have fixed  a defining polynomial $f\colon \AtR\to \RR$ of $\AAA$
and an orientation $\sigmaA$ of $\AtR$,
so that 
a standard orientation $\sigma_C$  is defined for each $C\in \Chamb$.
We calculate the capping hemispheres $H_{C, P}:=H_{C,\sigma_C,  P}$ explicitly.
%
%
\subsection{Good systems of local coordinates}\label{subsec:goodsystem}
Recall from Section~\ref{subsec:goodcoordinates}
that an affine coordinate system $(\xi, \eta)$ of $\AtC$
is \emph{good} if it is compatible with the $\RR$-structure of $\AtC$ and 
the affine coordinates $(x, y):=(\Real \xi, \Real \eta)|\AtR$ of $\AtR$ are compatible with the fixed orientation 
$\sigmaA$.
We refine this notion  to 
the notion of \emph{good local coordinate systems} on various 
 (real and complex) surfaces.
  \par
  Let $P$ be a point in $\SingBC=\Sing\BC$.
  We say that good coordinates $(\xi, \eta)$ of $\AtC$ are \emph{good at $P$}
  if $P$ is the origin $(0,0)$
  and 
  the two lines in $\Arr$ passing through $P$ are defined in $\AtR$ by 
\[
\ell_{a}(\RR) \colon y=a x
\quad\textrm{and}\quad 
\ell_{b}(\RR) \colon y=bx
\]
by some real numbers $a, b$ satisfying  $a<b$.
\begin{remark}\label{rem:arbitrarypair}
For any pair $(a\sprime, b\sprime)$ of real numbers satisfying $a\sprime<b\sprime$, 
we have good coordinates $(\xi, \eta)$ at $P$
such that the pair $(a, b)$ defined above is equal to the given pair $(a\sprime, b\sprime)$.
\end{remark}
 Next we  define local coordinates $(\tilxi, \mu)$ of $\YC$ and  local coordinates $(\tilx, m)$  of $\YR=\betash \AtR$.
 Let $(\xi, \eta)$ be affine coordinates of $\AtC$ good at $P\in \PPP$ as above.
 Let $r$ be a positive real number such that 
the open subset 
\[
U_r:=\set{(\xi, \eta)}{|\xi|<r, \;\; |\eta|<r}
\]
of $\AtC$ 
intersects no lines of $\ArrC$ other than $\ell_a(\CC)$ and $\ell_b(\CC)$:
\begin{equation}\label{eq:Ur}
\parbox{10cm}{for $\ell(\RR)\in \Arr$, if $\ell(\CC)\cap U_r\ne \emptyset$, then 
we have $\ell(\RR)\in \{\ell_a(\RR), \ell_b(\RR)\}$. }
\end{equation}
Then there exists a unique  coordinate system  $(\tilxi, \mu)$ on $\beta\inv(U_r)\subset \YC$ 
such that $\beta\colon \YC\to \AtC$ is given by
\begin{equation}\label{eq:betabytilximu}
(\tilxi, \mu)\mapsto (\xi, \eta)=(\tilxi, \tilxi\mu).
\end{equation}
In the chart $\beta\inv(U_r)$ of $(\tilxi, \mu)$, 
the exceptional curve $E_P$ is defined by $\tilxi=0$,  and 
the strict transforms $\betash\ell_a(\CC)$ and $\betash\ell_b(\CC)$ of $\ell_a(\CC)$ and $\ell_b(\CC)$ 
are defined by $\mu=a$ and $\mu=b$, respectively.
We put
\[
\tilxi=\tilx +\thei \tilu, \quad
\mu=m+\thei s, 
\]
where $\tilx$, $\tilu$, $m$, $s$  are real-valued functions on $\beta\inv(U_r)$.
Then we have 
\[
\YR=\{\tilu=s=0\}
\]
in $\beta\inv(U_r)$, 
and we can regard 
the restrictions of $(\tilx, m)$ to $\YR$ as local coordinates  of the real surface $\YR$.
From now on,  we consider $(\tilx, m)$ as   local coordinates  of $\YR$ by restriction.
\begin{definition}\label{def:goodsystems}
The local coordinate systems 
$(\xi, \eta)$ on $\AtC$, $(\tilxi, \mu)$ on $\YC$,  $(x, y)$ on $\AtR$,  and $(\tilx, m)$ on $\YR$, 
are called \emph{good local coordinate systems at $P$}.
\end{definition}
\subsection{The orientation of \texorpdfstring{$\YR$}{Y(R)}}
Let $\betaRR\colon \YR \to \AtR$ denote  the restriction of $\beta$ to $ \YR$,
which is a local isomorphism
 at every point of
\[
 \YR\spcirc:=\beta\inv(\AtR\setminus \SingBC). 
\]
Hence the fixed orientation $\sigmaA$ of $\AtR$
induces an orientation on $ \YR\spcirc$ via $\betaRR$,
which we denote by $\betaRR^*\sigmaA$.
\begin{lemma}\label{lem:betaorientation}
Let $Q$ be a point of $\beta\inv (U_r)\cap  \YR\spcirc$.
If $\tilx>0$ at $Q$,
then the ordered pair $(\partial/\partial \tilx, \partial/\partial m)$ of tangent vectors of 
$\YR$ at $Q$ is positive with respect to $\betaRR^*\sigmaA$,
whereas if  $\tilx<0$ at $Q$, 
then  $(\partial/\partial \tilx, \partial/\partial m)$  is negative with respect to $\betaRR^*\sigmaA$.
\end{lemma}
\begin{proof}
The map $\betaRR\colon \YR \to \AtR$ is given by 
$(\tilx, m)\mapsto (x, y)=(\tilx, \tilx m)$.
We can calculate $\betaRR^*\sigmaA$  by this formula.
\end{proof}
The restriction of the function $\mu$ to $E_P$ is an affine coordinate of  $E_P$.
Let $Q_{\infty}\in E_P$ be the point $\mu=\infty$ of $E_P$.
We consider the $m$-axis
\[
(E_P\setminus\{Q_{\infty}\})\cap \YR =\{\tilx=0\}
\]
in the chart $\beta\inv (U_r)\cap  \YR$  of $(\tilx, m)$.
%
\begin{corollary}\label{cor:downward}
The orientation $\betaRR^*\sigmaA$ on $\YR\spcirc$ induces 
the downward orientation on the $m$-axis $(E_P\setminus\{Q_{\infty}\})\cap \YR$,
that is, the direction that $m$ decreases is positive with respect to the orientation on the $m$-axis 
induced by $\betaRR^*\sigmaA$.
\end{corollary}
\begin{proof}
See Figure~\ref{fig:downward} of $\YR$,
in which $\partial/\partial \tilx$ is rightward 
and $\partial/\partial m$ is upward.
The orientation $\betaRR^*\sigmaA$ on $\YR\spcirc$  is 
 counter-clockwise 
 $\circlearrowleft$  in the region $\tilx>0$, and 
is  clockwise 
$\circlearrowright$ 
in the region $\tilx<0$.
Both of them induce a downward orientation on the $m$-axis.
\end{proof}
\begin{figure}
\begin{tikzpicture}[x=6mm, y=6mm]
\draw[->,thick](4,4)--(4,2);
\draw[thick](4,2)--(4,0);
\node at (2.7, 4.5) [right] {$E_P\cap \YR$};
\node at (1.8, 2) [right] {{\Huge$\circlearrowright$}};
\node at (5, 2) [right] {{\Huge$\circlearrowleft$}};
\draw[->](9,1)--(10,1);
\draw[->](9,1)--(9,2);
\node at (10,1) [right] {$\partial/\partial \tilx$};
\node at (9,2) [above] {$\partial/\partial m$};
\node at (-3,1) [right] {\phantom{a}};
\end{tikzpicture}
\caption{Orientation $\betaRR^*\sigmaA$}
\label{fig:downward}
\end{figure}
\subsection{Capping hemispheres}\label{subsec:capping}
Let $C$ be a bounded chamber such that $P\in \Vertexes (C)$.
We use 
the good local coordinate systems at $P$ given in the previous section.
We calculate the capping hemisphere $H_{C, P}=H_{C, \sigma_C, P}$
explicitly.
We have the following  four cases:
locally around $P$ in $\AtR$, the chamber $C$ is equal to the region 
\begin{equation}\label{eq:Ccases}
\renewcommand{\arraystretch}{1.2}
\begin{array}{cl}
\textrm{Case (1):} & \textrm{$ax \le y\le b  x$,} \\
\textrm{Case (2):} & \textrm{$ax \le y$ and $ b  x\le y$,} \\
\textrm{Case (3):} & \textrm{$ax \ge y\ge b  x$,} \\
\textrm{Case (4):} & \textrm{$ax \ge y$ and $ b  x\ge y$.} \\
\end{array}
\end{equation}
See Figure~\ref{fig:chambercases}.
\begin{remark}\label{rem:fourcases}
We can choose, without loss of generality, 
good coordinates $(\xi, \eta)$ of $\AtC$ at $P$ such that 
Case (1) occurs.
\end{remark}
\begin{figure}
\begin{tikzpicture}
\draw (0,0)--(5,2);
\draw (0,2)--(5,0);
\fill [black] (2.5,1) circle [x radius=.07, y radius=.07, rotate=0];
\node at (2.9,1)  [right] {$P$};
\node at (4.4,1)  [right] {(1)};
\node at (2.2,2) [right] {(2)};
\node at (0.5,1) [right] {(3)};
\node at (2.2,0) [right] {(4)};
\node at (5.3,-.2)  [right] {$\ell_a(\RR)$};
\node at (5.3,2) [right] {$\ell_b(\RR)$};
\end{tikzpicture}
\caption{Location of chambers}\label{fig:chambercases}
\end{figure}
In terms of the local coordinates $(\tilx, m)$ of $\YR$,
the closed region $\betash C$ is given in $\YR$ as follows:
\begin{equation*}\label{eq:Ccases2}
\renewcommand{\arraystretch}{1.2}
\begin{array}{cl}
\textrm{Case (1):} & \betash C=\set{(\tilx, m)}{\tilx\ge 0, \;\;a\le m\le b}, \\
\textrm{Case (2):} & \betash C=\set{(\tilx, m) }{(\tilx\ge 0 \;\rmand\; b\le m) \;\rmor\; (\tilx\le 0 \;\rmand\; m\le a)},  \\
\textrm{Case (3):} &  \betash C=\set{(\tilx, m) }{\tilx\le 0, \;\;a\le m\le b},  \\
\textrm{Case (4):} & \betash C=\set{(\tilx, m)}{(\tilx\le 0 \;\rmand\; b\le m) \;\rmor\; (\tilx\ge 0 \;\rmand\; m\le a)}. 
\end{array}
\end{equation*}
See Figure~\ref{fig:betashC}.
\begin{figure}
\begin{tikzpicture}[x=5.25mm, y=5.25mm]
\draw[fill=gray!30, line width=.0pt](3,2)--(6,2)--(6,4)--(3,4);
\draw[->,thick](3,6)--(3,5);
\draw[->,thick](3,5)--(3,3);
\draw[->,thick](3,3)--(3,1);
\draw[thick](3,1)--(3,0);
\draw[thick](0,4)--(6,4);
\draw[thick](0,2)--(6,2);
\node at (1.55, 6.5)  [right] {$E_P\cap \YR$};
\node at (3.1, 2.45)  [right] {$Q_a$};
\fill [black] (3,2) circle [x radius=.1, y radius=.1, rotate=0];
\node at (3.1, 4.45)  [right] {$Q_b$};
\fill [black] (3,4) circle [x radius=.1, y radius=.1, rotate=0];
\node at (6.3, 4)  [right] {$\betash\ell_b (\CC)$};
\node at (6.3, 2)  [right] {$\betash\ell_a (\CC)$};
\node at (2.2,-1)  [right] {Case (1)};
\end{tikzpicture}
\qquad\qquad
\begin{tikzpicture}[x=5.25mm, y=5.25mm]
\draw[fill=gray!30, line width=.0pt](6,4)--(3,4)--(3,6)--(6,6);
\draw[fill=gray!30, line width=.0pt](0,0)--(3,0)--(3,2)--(0,2);
\draw[->,thick](3,6)--(3,5);
\draw[->,thick](3,5)--(3,3);
\draw[->,thick](3,3)--(3,1);
\draw[thick](3,1)--(3,0);
\draw[thick](0,4)--(6,4);
\draw[thick](0,2)--(6,2);
\node at (1.55, 6.5)  [right] {$E_P\cap \YR$};
\node at (3.1, 2.45)  [right] {$Q_a$};
\fill [black] (3,2) circle [x radius=.1, y radius=.1, rotate=0];
\node at (3.1, 4.45)  [right] {$Q_b$};
\fill [black] (3,4) circle [x radius=.1, y radius=.1, rotate=0];
\node at (6.3, 4)  [right] {$\betash\ell_b (\CC)$};
\node at (6.3, 2)  [right] {$\betash\ell_a (\CC)$};
\node at (2.2,-1)  [right] {Case (2)};
\end{tikzpicture}
\vskip .7cm
\begin{tikzpicture}[x=5.25mm, y=5.25mm]
\draw[fill=gray!30, line width=.0pt](0,2)--(3,2)--(3,4)--(0,4);
\draw[->,thick](3,6)--(3,5);
\draw[->,thick](3,5)--(3,3);
\draw[->,thick](3,3)--(3,1);
\draw[thick](3,1)--(3,0);
\draw[thick](0,4)--(6,4);
\draw[thick](0,2)--(6,2);
\node at (1.55, 6.5)  [right] {$E_P\cap \YR$};
\node at (3.1, 2.45)  [right] {$Q_a$};
\fill [black] (3,2) circle [x radius=.1, y radius=.1, rotate=0];
\node at (3.1, 4.45)  [right] {$Q_b$};
\fill [black] (3,4) circle [x radius=.1, y radius=.1, rotate=0];
\node at (6.3, 4)  [right] {$\betash\ell_b (\CC)$};
\node at (6.3, 2)  [right] {$\betash\ell_a (\CC)$};
\node at (2.2,-1)  [right] {Case (3)};
\end{tikzpicture}
\qquad\qquad
\begin{tikzpicture}[x=5.25mm, y=5.25mm]
\draw[fill=gray!30, line width=.0pt](0,4)--(3,4)--(3,6)--(0,6);
\draw[fill=gray!30, line width=.0pt](3,2)--(6,2)--(6,0)--(3,0);
\draw[->,thick](3,6)--(3,5);
\draw[->,thick](3,5)--(3,3);
\draw[->,thick](3,3)--(3,1);
\draw[thick](3,1)--(3,0);
\draw[thick](0,4)--(6,4);
\draw[thick](0,2)--(6,2);
\node at (1.55, 6.5)  [right] {$E_P\cap \YR$};
\node at (3.1, 2.45)  [right] {$Q_a$};
\fill [black] (3,2) circle [x radius=.1, y radius=.1, rotate=0];
\node at (3.1, 4.45)  [right] {$Q_b$};
\fill [black] (3,4) circle [x radius=.1, y radius=.1, rotate=0];
\node at (6.3, 4)  [right] {$\betash\ell_b (\CC)$};
\node at (6.3, 2)  [right] {$\betash\ell_a (\CC)$};
\node at (2.2,-1)  [right] {Case (4)};
\end{tikzpicture}
\caption{$\betash C$ on $\YR$}\label{fig:betashC}
\end{figure}
Recall that the open subset $U_r$ of $\AtC$ is defined in such a way that~\eqref{eq:Ur} holds.
Hence the  pullback $ \beta^* f$ of 
the defining polynomial $f\colon \AtC\to \CC$ of $\Arr$ by $\beta$
is written as 
 \[
 \beta^* f=u_P\cdot \tilxi^2\cdot (\mu-a)\cdot (\mu-b)
 \]
 in the chart $\beta\inv(U_r)$ of $(\tilxi, \mu)$, 
 where $u_P$ is a complex-valued continuous function on $\beta\inv(U_r)$ that
 has \emph{no zeros},  takes values in $\RR$ on $\beta\inv(U_r)\cap \YR$, and is constant on $\beta\inv(U_r)\cap E_P$.
 We denote by  
\[
c_P \in \RR\setminus \{0\}
\]
 the value of $u_P$ on $E_P$.
Note that $\beta\inv(U_r)$ is simply connected.
Hence we can define 
a complex-valued continuous function $v_P$  on $\beta\inv(U_r)$   such that  $v_P^2=u_P$ and that 
\begin{equation}\label{eq:uP}
v_P|(\beta\inv(U_r)\cap \YR)  \textrm{\;\;takes values in }
\begin{cases}
 \posR & \textrm{if $c_P \in \posR$,}\\
\thei\posR & \textrm{if $c_P \in \negR$.}
\end{cases}
\end{equation}
Recall from Section~\ref{subsec:StandardOrientations}  that $\omega$ is the function on $W$ such that 
the  double covering $\pi\colon W\to \AtC$ is given by $\omega^2=f$.
We regard $\omega$ as a function on $X$ by $\rho\colon X\to W$.
We also regard $v_P$, $\tilxi$, $\mu$ as functions on 
$\phi\inv (\beta\inv(U_r))\subset X$ by $\phi\colon X\to \YC$.
We then put 
\begin{equation}\label{eq:defzeta}
 \zeta:= \frac{\omega}{v_P \cdot \tilxi \cdot (\mu -b)},
\end{equation}
 which is a  meromorphic function on $\phi\inv (\beta\inv(U_r))\subset X$.
 Then we have
\begin{equation}\label{eq:zetasq}
 \zeta^2=\frac{\mu-a}{\mu-b},
\end{equation}
 which gives the double covering $\phi\colon X\to \YC$.
 We put
\[
z:=\zeta | D_P.
\]
Let $Q_a$ (resp.~$Q_b$) be the intersection point
of $E_P$ and $\betash\ell_a(\CC)$ (resp.~$\betash\ell_b(\CC)$).
The restriction $\phi|D_P\colon D_P\to E_P$  of $\phi$ to $D_P\subset X$ is a double covering whose branch points are $Q_a$ and $Q_b$.
Let $\tilQ_a\in D_P$   (resp.~$\tilQ_b\in D_P$)  be the point of $D_P$ lying over $Q_a$ (resp.~$Q_b$). 
Then $z$ is an affine coordinate of the Riemann sphere $D_P$
such that  $\tilQ_a$ and $\tilQ_b$ are  given by $z=0$ and $z=\infty$, respectively.
 The main result of this section is as follows:
  \begin{table}
 \[
 \renewcommand{\arraystretch}{1.4}
 \begin{array}{cc|c}
C  & f(C\spcirc)  & H_{C,  P}  \\
\hline 
\textrm{Case } (1) & f(C\spcirc) \subset \posR &   \Real z\le 0 \\
\textrm{Case } (1) & f(C\spcirc) \subset \negR &   \Real z\ge 0 \\
\textrm{Case } (2) & f(C\spcirc) \subset \posR &   \Imag z\le 0 \\
\textrm{Case } (2) & f(C\spcirc) \subset \negR &   \Imag z\le 0 \\
\textrm{Case } (3) & f(C\spcirc) \subset \posR &   \Real z\ge 0 \\
\textrm{Case } (3) & f(C\spcirc) \subset \negR &   \Real z\le 0 \\
\textrm{Case } (4) & f(C\spcirc) \subset \posR &   \Imag z\ge 0 \\
\textrm{Case } (4) & f(C\spcirc) \subset \negR &   \Imag z\ge 0 
 \end{array}
 \]
 \vskip .7mm
 \caption{Capping hemispheres}\label{table:hemispheres}
 \end{table}
 \begin{proposition}\label{prop:cappinghemispheres}
 The capping hemisphere $H_{C,  P}$ is given as in Table~\ref{table:hemispheres}.
 \end{proposition}
 \begin{proof}
To ease the notation,
 we put
 \[
 \rP:=\nonnegR\cup\{\infty\}, 
 \quad \iP:=\thei \nonnegR\cup\{\infty\}.
 \]
 Recall that $\phi\inv(\beta\inv (C\spcirc))=\rho\inv (\pi\inv (C\spcirc))$ has two connected components,
 which we call sheets.
 We denote  by $(\phi\inv\beta\inv C\spcirc)_+$
 the sheet on which we have 
 \begin{equation}\label{eq:omega}
 \begin{cases}
 \omega \in \rP \;\textrm{holds}& \textrm{if $f\in P$ on $C$}, \\
 \omega \in \iP  \;\textrm{holds}& \textrm{if $f\in -P$ on $C$}.
 \end{cases}
 \end{equation}
Recall~Definition~\ref{def:sigmaC}  of the standard orientation $\sigma_C$.
The sheet $(\phi\inv\beta\inv C\spcirc)_+$ is the positive-sheet,  
that is, 
the restriction of 
  $\phi$  to $(\phi\inv\beta\inv C\spcirc)_+$ is an orientation-preserving isomorphism 
 from $(\phi\inv\beta\inv C\spcirc)_+$ with orientation $\sigma_C$ to
the open subset  $\beta\inv (C\spcirc)$ of $\YR$ with orientation $\betaRR^*\sigmaA$.
 \par
 By the locations of $C$ and the signs of $f(C)$,
 we have eight cases,
 which are given by the rows of Table~\ref{table:hemispheres}.
 For each case, 
 we calculate values at points of $(\phi\inv\beta\inv C\spcirc)_+$
 of the functions
 \begin{align*}
 g & := (m-a)(m-b), \\
 u_P & \phantom{:}=\beta^*f/(\tilx^2 g), \\
v_P  &\quad  \textrm{as defined by~\eqref{eq:uP}}, \\
 h & := \, \tilx \cdot (m-b), \\
 \omega &\quad  \textrm{as defined by~\eqref{eq:omega}}, \\
 \zeta &  = \omega/(v_P  h).
 \end{align*}
The results are given in Table~\ref{table:functions}.
The intersection 
\[
K_{C, P}:=D_P \cap \overline{(\phi\inv\beta\inv C\spcirc)_+}
\]
of $D_P$ and the closure of $(\phi\inv\beta\inv C\spcirc)_+$ in $X$
is an arc on $D_P$  connecting the ramification points $\tilQ_a$ and $\tilQ_b$ of $\phi|D_P\colon D_P\to E_P$.
In terms of the parameter $z$ on $D_P$,
the arc $K_{C, P}$ is the closed arc connecting 
$z=0$ and $z=\infty$ 
along the closure of the range of the function $\zeta$ on $(\phi\inv\beta\inv C\spcirc)_+$.
Since we have calculated this range above,
we can describe  
$K_{C, P}$ in terms of $z$.
See the column~$K_{C, P}$ of Table~\ref{table:functions}.
Since the deck transformation of 
the double covering $\phi|D_P\colon D_P\to E_P$ is
given by $z\mapsto -z$,
we have
\begin{equation}\label{eq:KcupmK}
S_{C, P}=D_P\cap \pants(C) = K_{C, P}\cup (-K_{C, P}).
\end{equation}
The covering $\phi|D_P$
maps $K_{C, P}$
to 
$J_{C, P}=E_P\cap \betash C$,
which  is a segment of the $m$-axis with 
$Q_{\infty}$ 
being  added.
The orientation $\sigma_C$ on $(\phi\inv\beta\inv C\spcirc)_+$  induces an orientation on $K_{C, P}$,
and hence on $\phi(K_{C, P})=J_{C, P}$.
Since $\phi$ induces an orientation-preserving isomorphism from 
$(\phi\inv\beta\inv C\spcirc)_+$ with orientation $\sigma_C$ to 
$\beta\inv (C\spcirc)$ with orientation $\betaRR^*\sigmaA$, 
Corollary~\ref{cor:downward} implies that 
this orientation  on $J_{C, P}$ is downward, that is, 
 \begin{align*}
 Q_b \searrow Q_a & \quad \textrm{in Cases (1) and (3)}, \\
 Q_a \searrow Q_{\infty}\searrow Q_b &\quad \textrm{in Cases (2) and (4)}.
 \end{align*}
See Figure~\ref{fig:betashC}.
In the column $\ori_K$ of Table~\ref{table:functions}, this orientation is expressed in terms of 
the parameter $z$.
Combining the computations of $K_{C, P}$ and $\ori_K$,
we obtain the orientation $ \ori_C$ 
on the circle $S_{C, P}$ induced by 
the orientation $\sigma_C$ of $\pants(C)$ by~\eqref{eq:KcupmK}.
In Table~\ref{table:functions},
  the circle $S_{C, P}$ and the orientation $ \ori_C$  are given as follows:
 \begin{align*}
  \downarrow\;\;\; &\textrm{means $S_{C, P}=\{\Real z=0\}$ and the orientation is downward}, \\
  \uparrow\;\;\; &\textrm{means $S_{C, P}=\{\Real z=0\}$ and the orientation is upward}, \\
  \rightarrow\;\;\;  &\textrm{means $S_{C, P}=\{\Imag z=0\}$ and the orientation is rightward}, \\
  \leftarrow\;\;\;  &\textrm{means $S_{C, P}=\{\Imag z=0\}$ and the orientation is leftward}.
\end{align*}
The orientation $\ori_H$ on $S_{C, P}$ given by the complex structure of the  capping hemisphere $H_{C,  P}$
is the opposite of $ \ori_C$.
Thus we obtain $H_{C,  P}$ as in the last column of Table~\ref{table:functions}.
 \begin{table}
 {\footnotesize
 \[
 \renewcommand{\arraystretch}{1.4}
 \begin{array}{cc|cccccc|cccc|c}
C  & \beta^*f  &g& u_P & v_P& h& \omega  &\zeta & K_{C, P} &\ori_K&  \ori_C & \ori_H &H_{C, P}\\
\hline 
  (1) &+ & -& -&  \iP& -& \rP  &\iP &\iP &\infty\to 0   &\downarrow &\uparrow &\Real z \le 0\\
  (1) & - & -& +&   \rP& -& \iP &-\iP &-\iP &\infty\to 0 & \uparrow &\downarrow &\Real z \ge 0  \\
  (2) &+& +& +&   \rP& +& \rP &\rP &\rP&0 \to \infty&\rightarrow &\leftarrow &\Imag z\le 0 \\
  (2) &- &  +& -&   \iP&+&  \iP &\rP &\rP&0 \to \infty&\rightarrow &\leftarrow &\Imag z\le 0 \\
  (3) &+ & -& -&  \iP&  +& \rP &-\iP&-\iP&\infty\to 0&  \uparrow &\downarrow &\Real z \ge 0 \\
  (3) &- &  -& +&  \rP&+&  \iP &\iP&\iP&\infty\to 0 &\downarrow  & \uparrow &\Real z \le 0\\
  (4) &+& +&+&  \rP&-& \rP &-\rP&-\rP &0 \to \infty &\leftarrow  & \rightarrow  &\Imag z\ge 0\\
  (4) &- & +& - &  \iP&-& \iP &-\rP &-\rP&0 \to \infty &\leftarrow  & \rightarrow &\Imag z\ge 0 \\
 \end{array}
 \]
 }
 \vskip .9mm
 \caption{Functions on $(\phi\inv\beta\inv C\spcirc)_+$}\label{table:functions}
 \end{table}
\end{proof}
\begin{corollary}\label{cor:coherent}
The collection $\set{\sigma_C}{C\in \Chamb}$ 
of standard orientations is coherent in the sense of Definition~\ref{def:coherent}.
\end{corollary}
\begin{proof}
Suppose that  $C\cap C\sprime$ consists of a single point $P$.
Then the signs of $f\res{C}$ and of $f\res{C\sprime}$ are the same.
Looking at Cases~(1) and~(3) (or Cases~(2) and~(4)) of Table~\ref{table:hemispheres},
we see that $H_{C, P}$ and $H_{C\sprime, P}$ are distinct.
\end{proof}
Now Proposition~\ref{prop:coherentcollection} follows immediately from Corollary~\ref{cor:coherent}. 
\qed
\section{Displacement}\label{sec:displacements}
We  construct two types of displacements of vanishing cycles $\stdvan(C)$ in $X$.
One is called  \emph{$C$-displacements},
and the other is called  \emph{$E$-displacements}.
\subsection{Notation and terminology about displacements}\label{subsec:dislpcements}
For a small positive real number $\varee$,  let $\intval$ denote the closed interval $[0, \varee]\subset \RR $.
\begin{definition}\label{def:displacement}
A  \emph{displacement} of a subspace $S$ of a topological space $T$ 
is a continuous map 
\[
d\colon S\times\intval \to T
\]
such that
\begin{itemize}
\item 
 $d(P,0)=P$ for any $P\in S$, and 
 \item the restriction $d_t:=d|(S\times\{t\})$ 
 of $d$ to $S\times\{t\}$  is a homeomorphism from $S$ to  its image $S_t:=d(S\times\{t\})$
 for any $t\in\intval$.
 \end{itemize}
 \end{definition}
We sometimes write 
\[
\set{S_t}{t\in\intval} \quad\textrm{or}\quad \set{d_t}{t\in\intval}   \quad\textrm{or}\quad  d_t\colon S\to T\;\;(t \in\intval)
\]
to denote  the displacement  $d\colon S\times\intval \to T$.
By further abuse of notation,
we often say that \emph{$S_{\varee}$ is a displacement of $S$ in $T$}.
\par
If $S=S_0$ is equipped with an orientation $\sigma_0$,
then a displacement  $S_{\varee}$ of $S$  is also oriented by $\sigma_0$ via the homeomorphism  $d_{\varee}\colon S\to S_{\varee}$.
If $S$ is a topological cycle, 
then  $S_{\varee}$ is also a topological cycle,
and their classes  in the homology group of $T$
are the same.
\par
We say that a displacement $d_t\colon S\to T$ \emph{preserves a subspace  $R$ of $T$}
if $d_t(Q) \in R$ holds  for any $Q\in S\cap R$ and any $t \in\intval$.
We say that $d$ is \emph{stationary} on a subspace $S\sprime \subset S$ if $d(Q, t)=Q$  holds
for any $(Q, t)\in S\sprime\times  \intval$.
A \emph{trivial displacement} is a displacement  that is stationary on the whole space $S$.
\subsection{\texorpdfstring{$C$}{C}-displacement}\label{subsec:Cdislpcements}
Recall from Section~\ref{subsec:translations} that $T(\AtR)$ is the $\RR$-vector space 
of translations of  $\AtR$,
and that $T[\lambda(\RR)]$ is the subspace of translations 
preserving a real affine line $\lambda(\RR)\subset \AtR$.
Let  $C$ be a bounded chamber, and let 
\[
\delta\colon C \to T(\AtR)
\]
be  a continuous function.
\begin{definition}\label{def:edgecondition}
Suppose that 
$\ell_i(\RR)\in \Arr$ defines an edge $C\cap \ell_i(\RR)$  of $C$.
We say that $\delta$ satisfies the \emph{$e$-condition} for the edge  $C\cap \ell_i(\RR)$ if 
 $\delta(Q)\in \Tl{\ell_i(\RR)}$ holds for any $Q\in C\cap \ell_i(\RR)$.
 \end{definition}
\begin{remark}
If $\delta$ satisfies the $e$-condition for every edge of $C$, 
then we have $\delta(P)=0$ 
for every  $P\in \Vertexes  (C)$, 
because $\Tl{\ell_i(\RR)}\cap \Tl{\ell_j(\RR)}=\{0\}$ holds  if $\ell_i(\RR)\cap\ell_j(\RR)=\{P\}$.
\end{remark}
\par
Suppose that  $\delta\colon C \to T(\AtR)$ 
satisfies the $e$-condition for every edge of  $C$.
We  define $d\colon C\times\intval \to   \AtC$ by 
\begin{equation}\label{eq:dbydelta}
d(Q, t):=Q+\thei \delta(Q) t.
\end{equation}
Note that we have 
\begin{equation}\label{eq:pRd}
\prR(d(Q, t))=Q\;\;\textrm{for any}\;\; (Q, t)\in C\times\intval,
\end{equation}
where $\prR\colon \AtC\to \AtR$ is the projection 
that takes the real part of points of $\AtC$ (see Section~\ref{subsec:translations}).
This implies that  $d_t:=d|(C\times\{t\})$ is a homeomorphism from $C$ to  
\[
C_t:=d_t(C)
\]
for any $t\in \intval$, and hence 
$d$ is a displacement of $C$ in $\AtC$.
By the $e$-condition, 
the displacement  $d$ preserves $\ell_i(\CC)$ 
for each $\ell_i(\RR)\in \Arr$ defining edges of $C$.
Since $\SingBC \subset \AtR$ and $\prR(\BC)=\BR$,
 we see from~\eqref{eq:pRd} that
\begin{align}
\label{eq:CspbcircIsPreserved}
& Q\in C\spbcirc =C\setminus \Vertexes  (C)\;\; \Longrightarrow\;\; d(Q, t)\notin \SingBC, \\
\label{eq:CspcircIsBCreserved}
& Q\in C\spcirc =C\setminus (C\cap \BR) \;\; \Longrightarrow\;\; d(Q, t)\notin \BC.
\end{align}
Note that  $\beta\colon \YC\to\AtC$ induces an isomorphism from $\beta\inv( C\spbcirc)$ to $C\spbcirc$.
By~\eqref{eq:CspbcircIsPreserved}, 
there exists a unique  continuous map
 \[
 d\spbcirc \colon\beta\inv( C\spbcirc) \times\intval\to \YC
 \] 
 that fits in the commutative diagram
\begin{equation}\label{eq:dspbcirc}
 \renewcommand{\arraystretch}{1.2}
 \begin{array}{ccc}
\beta\inv( C\spbcirc) \times\intval & \maprightsp{d\spbcirc} & \YC \\
 \mapdownleft {\beta\times \id} && \mapdownright {\beta}\\
  C \times\intval & \maprightsb{d} & \AtC.
 \end{array}
 \end{equation}
 It is obvious that  $d\spbcirc$  is a displacement of $\beta\inv( C\spbcirc)$ in $\YC$.
 We define a map 
 \[
d\spsh \colon \betash C \times\intval\to \YC
 \]
 by the following:
 \[
d\spsh (Q, t):=\begin{cases}
   d\spbcirc(Q, t) & \textrm{if $Q\in  \beta\inv( C\spbcirc)$, }\\
   Q & \textrm{if $Q\in  \betash C\setminus \beta\inv( C\spbcirc)=\bigcup_{P\in \Vertexes (C)}  J_{C, P}$. }
  \end{cases}
 \]
\begin{definition}\label{def:Cdelta1}
We say that $\delta\colon C\to T(\AtR)$ is \emph{regular at $P\in \Vertexes  (C)$}
if $d\spsh$ is continuous at $(Q, t)$
for any  $Q\in J_{C, P}$ and any $t\in \intval$.
We simply say that $\delta$ is \emph{regular} if 
$\delta$ is regular at every vertex $P$ of $C$,
that is, if $d\spsh$ is continuous.
\end{definition}
If $\delta$ is regular,
then $d\spsh$ is  a displacement  of $\betash C$ in $\YC$ 
that is stationary on each $J_{C, P}$, and 
that is preserving  all  $\betash\ell_i(\CC)\subset \betash\YC$,
where $\ell_i(\RR)\in \Arr$ defines an edge of $C$.
In particular, the displacement $d\spsh$ preserves the intersection $\betash C\cap \betash\BC$.
Since 
the branch locus of $\phi\colon \XC\to \YC$ is $\betash \BC$, 
it follows that  $d\spsh$ lifts to  a displacement 
\[
d\sp{\pants} \colon \pants (C) \times \intval \to X
\]
of $\pants (C)$ in $X$
that is stationary on each $S_{C, P}\subset \bdr \pants (C)$.
We have 
\[
 \pants (C)_t=\phi\inv((\betash C)_t)
\]
with the orientation given by the standard orientation $\sigma_C$ of $\phi\inv(\beta\spsh C)$.
Gluing $d\sp{\pants}$ with the trivial displacements of the capping hemispheres $H_{C, P}$,
we obtain a displacement 
\[
d\sp{\stdvan} \colon \stdvan (C) \times \intval \to X
\]
of $\stdvan (C)$ in $X$
that is stationary on each $H_{C, P}$.
\begin{definition}\label{def:Cdelta2}
These displacements 
 $d$, $d\spsh$, $d\sp{\pants}$,  and $d\sp{\stdvan}$ 
 are called 
the  \emph{$C$-displacements associated with a regular continuous map  $\delta\colon C\to T(\AtR)$}
 satisfying the $e$-condition for every edge of $C$.
\end{definition}
We  introduce an operation  of \emph{$q_{\rho}$-modification},
which is useful in obtaining   a regular continuous map $\delta\colon C\to T(\AtR)$.
We use the good  local coordinate systems 
$(\xi, \eta)$ on $\AtC$, $(\tilxi, \mu)$  on  $\YC$,
$(x, y)$ on $\AtR$, and $(\tilx, m)$ 
on  $\YR$,
that are defined in Section~\ref{subsec:goodsystem}.
 We can assume,
 without loss of generality,  that  we are 
 in  Case (1) in~\eqref{eq:Ccases}, so that 
there exist real numbers $a, b$ with $a<b$ such that
$C$ is given by $ax\le y\le bx$ locally around  $P=(0,0)$ in $\AtR$.
We define real-valued functions 
$\delta_x(x, y)$ and $\delta_y(x, y)$
defined in  a small neighborhood of the origin $P$ in 
$\set{(x, y)}{ax\le y\le bx}$
by 
\begin{equation}\label{eq:deltaxdeltay}
\delta(Q)=\delta_x(x(Q), y(Q)) \cdot e_x + \delta_y(x(Q), y(Q))\cdot  e_y,
\end{equation}
where  $(x(Q), y(Q))$ are the coordinates of $Q\in C$,
and 
$e_x, e_y$ are the basis of $T(\AtR)$ given by~\eqref{eq:exey}.
Because $\beta$ restricted to $\YR$  is written as 
\[
(\tilx, m)\mapsto (x, y)=(\tilx, m \tilx), 
\]
the displacement $d\spbcirc\colon \beta \inv(C\spbcirc)\times\intval \to \YC$ defined by the diagram~\eqref{eq:dspbcirc}
 is written 
in terms of
the coordinate systems $(\tilx, m)$ and $(\tilxi, \mu)$ as 
\begin{equation}\label{eq:dspbcirc2}
\begin{split}
&((\tilx, m), t)\mapsto  \\
&\;\; (\tilxi, \mu)=\left(
\tilx+\thei \delta_x(\tilx, m \tilx) \cdot t,
\;\;
\frac{m\tilx+\thei \delta_y(\tilx, m \tilx) \cdot t}{\tilx+\thei \delta_x(\tilx, m \tilx)\cdot t}
\right).
\end{split}
\end{equation}
We introduce an inner-product on the $\RR$-vector space $T(\AtR)$.
Then we have a distance on $\AtR$.
For points $Q, Q\sprime\in \AtR$, we denote by $|QQ\sprime|$ the distance between $Q$ and $Q\sprime$.
Let $\rho$ be a small positive real number such that
$|P\Pp |>2\rho$ holds for any pair of distinct vertexes $P, \Pp$ of $C$.
For each vertex $P\in \Vertexes (C)$, we put
\begin{equation}\label{eq:UPrho}
U_{P,\rho}:=\set{Q\in C}{|P Q|\le \rho}.
\end{equation}
Then each point of $C$ belongs to at most one of $U_{P,\rho}$.
For $Q\in U_{P,\rho}$, let $q_{P,\rho} (Q)$ denote the unique point on
the line segment $PQ$ such that the distance $|P q_{P,\rho}  (Q)|$ between $P$ and $q_{P,\rho}  (Q)$ 
satisfies  
\[
|P q_{P,\rho}  (Q)|=|PQ|^2/\rho.
\]
Since $u\mapsto u^2/\rho$ is a bijection from the closed interval $[0, \rho]\subset \RR$ to itself,
we see that $q_{P,\rho}$ is a self-homeomorphism of $U_{P,\rho}\subset C$.
We construct $q_{\rho}\colon C\to C$ by 
\[
q_{\rho}(Q):=\begin{cases}
q_{P,\rho}(Q) & \textrm{if $Q\in U_{P,\rho}$ for some $P\in \Vertexes (C)$,}\\
Q & \textrm{otherwise.}
\end{cases}
\]
Since $q_{P,\rho} (Q)=Q$ for $Q\in U_{P,\rho}$ with $|PQ|=\rho$,
we see that $q_{\rho}$ is continuous on the whole $C$,  and hence is a self-homeomorphism of $C$.
\begin{definition}\label{def:qmodification}
Let $\delta\colon C\to T(\AtR)$ be a continuous map.
We call the composite $\delta\circ q_{\rho}\colon C\to T(\AtR)$ of $q_{\rho}$ and $\delta$
the \emph{$q_{\rho}$-modification} of $\delta$.
\end{definition}
Note that  $q_{\rho}\colon C\to C$ preserves each edge of $C$.
Hence, 
if $\delta$ satisfies the $e$-condition for every edge of  $C$,
then so does its $q_{\rho}$-modification $\delta\circ q_{\rho}$.
\begin{lemma}\label{lem:qmodification}
Let $\delta\colon C\to T(\AtR)$ be a  continuous map that satisfies the $e$-condition for every edge of  $C$.
Suppose that $\delta$  is affine-linear  in a small neighborhood of each vertex $P$ of $C$.
Then the  $q_{\rho}$-modification $\delta\circ q_{\rho}$ is regular.
\end{lemma}
\begin{proof}
We define continuous functions $\delta_x\circ q_{\rho}(x, y)$  and $\delta_y\circ q_{\rho}(x, y)$  on $U_{P, \rho}$ by 
\[
\delta \circ q_{\rho} (Q)=(\delta_x\circ q_{\rho})(x(Q), y(Q))  \cdot e_x + (\delta_y\circ q_{\rho})(x(Q), y(Q))\cdot  e_y.
\]
Note that the function $m$ in the coordinates $(\tilx, m)$  of $\YR$ is bounded in a neighborhood of $J_{C, P}$ in $\betash C$.
By the assumption that $\delta$ is affine-linear locally around $P$,
it follows that
\[
|\delta_x (\tilx, m\tilx )|/|\tilx| \quand |\delta_y (\tilx, m\tilx )|/|\tilx|
\]
are bounded in a neighborhood of  $J_{C, P}$ in $\betash C$.
On the other hand, since $|u^2/\rho|$ tends to $0$ 
faster than $|u|$ does as $u\to 0$,
we have 
\[
|(\delta_x\circ q_{\rho}) (\tilx, m\tilx )|/|\tilx|\;\to\; 0,  \quad |(\delta_y\circ q_{\rho})  (\tilx, m\tilx )|/|\tilx| \; \to\; 0, 
\]
as $\tilx\to 0$.
Hence we have 
\[
\frac{m\tilx+\thei (\delta_y\circ q_{\rho})(\tilx, m \tilx) 
\cdot t}{\tilx+\thei (\delta_x\circ q_{\rho})(\tilx, m \tilx)\cdot t}\;\; \to\;\;  m
\]
as $\tilx\to 0$.
Therefore  by~\eqref{eq:dspbcirc2} with $\delta_x$ and $\delta_y$ 
replaced by $\delta_x \circ q_{\rho}$ and $\delta_y\circ q_{\rho}$ respectively, 
we see that,
if $\delta$ is $q_{\rho}$-modified, then $d\spbcirc(Q, t)$ tends to $ Q$ as $Q\in \beta\inv (C\spbcirc)$ approaches  
a point of 
$J_{C, P}\subset \betash C$,
and hence $\delta\circ q_{\rho}$ is regular at $P$.
\end{proof}
 \subsection{\texorpdfstring{$E$}{E}-displacement}\label{subsec:Edisp}
Next we define \emph{$E$-displacements}
of various subspaces in $\YC$ and  $X$.
Let $C$ be a bounded chamber,
and $P$ a vertex of $C$.
\begin{definition}\label{def:tubular}
For a point $Q\in E_P$,
we denote by $\lambda_Q(\CC)$ the complex affine line in  $\AtC$ passing through $P$ 
such that its strict transform 
$\betash \lambda_Q(\CC)$ in $\YC$ intersects $E_P$ at $Q$.
A \emph{tubular neighborhood} of $E_P$  is an  open neighborhood $\NNN_P\subset \YC$ 
of $E_P$ in $\YC$ equipped with a continuous map  
\[
\projNN\colon \NNN_P \to E_P
\]
such that, 
for any $Q\in E_P$, 
the fiber $\projNN\inv (Q)$ is an open disk of 
$\betash \lambda_Q(\CC)$
with center $Q$.
We call $\projNN$ the \emph{projection} of the tubular neighborhood $\NNN_P$.
\end{definition}
Let $\projNN\colon \NNN_P \to E_P$ be a tubular neighborhood of $E_P$.
We consider the inclusion  $E_P\inj \NNN_P$ as a section of $\projNN$,
which we call the \emph{zero section}.
For $Q\sprime \in \projNN\inv (Q)$ and  $t\in I:=[0,1]$,
let $tQ\sprime$ denote the unique point on 
$\projNN\inv (Q)\subset \betash \lambda_Q(\CC)$
such that 
\[
\beta(tQ\sprime)=P+t\cdot\tau_{P, \beta(Q\sprime)},
\]
where $\tau_{P, \beta(Q\sprime)}\in T(\AtC)$ is the translation that maps $P$ to $\beta(Q\sprime)$.
\par
We describe a tubular neighborhood in terms of local coordinates.
Let $(\xi, \eta)$, $(\tilxi, \mu)$, $(x, y)$, $(\tilx, m)$ 
 be good  local coordinate systems of $\AtC$, $\YC$, $\AtR$, $\YR$, respectively, 
 given in Section~\ref{subsec:goodsystem}.
 We can assume, without loss of generality, 
 that the location of $C$ is  Case (1) in~\eqref{eq:Ccases},
 so that there exist real numbers $a, b$ with $a<b$ such that 
 $ \betash C$ is defined  
by $\tilx\ge 0$ and $ a\le m\le b$ in the chart of $(\tilx, m)$ on  $\YR$.
 By definition,  in terms of the coordinates $(\tilxi, \mu)$ of $\YC$, 
 the projection $\projNN$ is given by 
 \[
 \NNN_P\ni (\tilxi, \mu)\;\;\mapsto\;\; (0, \mu) \in E_P,
 \]
and the map $Q\sprime\mapsto t Q\sprime$ is given by
\[
(\tilxi, \mu) \;\;\mapsto\;\; (t \tilxi, \mu).
\]
%
\begin{definition}\label{def:admissiblesection}
A continuous section 
\[
s\colon E_P \to \NNN_P
\]
of $\projNN$ is said to be \emph{admissible with respect to $C$}  
if it satisfies  the following.
\begin{enumerate}[(s1)]
\item There exists one and only one point $Q_0\in E_P$ such that $s(Q_0)\in E_P$.
\item The point $Q_0\in E_P$ is not on  the arc $J_{C, P}=E_P\cap \beta\spsh C$.
\item If $Q\in J_{C, P}$, then $s(Q)\in \betash C$.
\end{enumerate}
\end{definition}
Suppose that $s\colon E_P\to \NNN_P$ is a section 
admissible with respect to $C$.
We can write 
$s$ as
\[
\mu\mapsto (\tilxi, \mu)=(\ss(\mu), \mu)
\]
in  a neighborhood of $J_{C, P}$ in $E_P$, 
where 
$\ss$ is a continuous function 
such that, if $m$ is a real number in the  closed interval $[a, b]$,
then $\ss(m)\in \RR_{>0}$.
 \begin{example}\label{example:constantsection}
Let $c$ be a sufficiently small positive real number.
Then we have 
a section 
$s\colon E_P\to \NNN_P$ admissible with respect to 
$C$ such that the function $\ss(\mu)$ is constantly equal to $c$ in a small neighborhood of $J_{C, P}$ in $E_P$. 
\end{example}
Note that  the image $s(E_P)$ of $s$ and $E_P$ intersect only at $Q_0$, 
and, since $s(E_P)$ and $E_P$ are  homologous in  $\YC$,  the local intersection number at $Q_0$ is 
equal to the self-intersection number $-1$ of $E_P$.
\par
For $t\in I$, 
we denote by
\[
ts\colon E_P\to \NNN_P\subset \YC
\]
the continuous map $Q\mapsto t\cdot s(Q)$.
Let $\vareep$ be a sufficiently  small positive real number.
For $t\in \intvalp$,
let $(E_P)_t\subset \NNN_P$ be the image of $E_P$ by $ts$,
and let $(J_{C, P})_t$ be the image of $J_{C, P}$ by $t s$.
Note that we have 
\[
(J_{C, P})_t=(E_P)_t\cap  \betash C.
\]
Let $Q_a$ (resp.~$Q_b$)  be  the intersection point of $E_P$
and 
$\betash\ell_a(\CC)$ (resp.~$\betash\ell_b(\CC)$).
Then $(J_{C, P})_t$ is an arc on $(E_P)_t$ connecting $ts (Q_a)$ and $ts(Q_b)$.
We put 
\[
(D_P)_t:=\phi\inv ((E_P)_t).
\]
Then  $\phi|(D_P)_t\colon (D_P)_t\to (E_P)_t$ is a double covering whose branch points
are the end points $ts(Q_a)$ and $ts(Q_b)$ of the arc $(J_{C, P})_t$.
Therefore 
\[
(S_{C, P})_t:=\phi\inv((J_{C, P})_t)
\]
is a circle on $(D_P)_t$.
Note that $(S_{C, P})_t$ decomposes $(D_P)_t$ into the union of two closed hemispheres.
Let $(H_{C, P})_t$ denote the hemisphere obtained from 
$(H_{C,  P})_0=H_{C,P}$ by continuity.
Thus, from the  section $s\colon E_P\to \NNN_P$
admissible with respect to $C$,  we obtain  displacements
\begin{align*}
\set{(J_{C, P})_t}{t \in \intvalp} & \textrm{ of $J_{C, P}$ in $\betash C$,}\\
\set{(E_P)_t}{t \in \intvalp} & \textrm{ of $E_P$ in $\YC$,}\\
\set{(D_P)_t}{t \in \intvalp} & \textrm{ of $D_P$ in $X$,}\\
\set{(S_{C, P})_t}{t \in \intvalp} & \textrm{ of $S_{C, P}$ in $\phi\inv(\betash C)$, and }\\
\set{(H_{C,P})_t}{t \in \intvalp} & \textrm{ of $H_{C, P}$ in $X$,}
\end{align*}
which we call \emph{$E$-displacements} associated with $s\colon E_P\to \NNN_P$.
\par
 For positive real numbers $\alpha, \beta$ with $\alpha<\beta$,
 let $g_{\alpha, \beta}\colon \RR_{\ge 0}\to  \RR_{\ge 0}$ denote
the continuous  function defined by
 \[
 g_{\alpha, \beta}(u):=
 \begin{cases}
\alpha+(\beta-\alpha) u/\beta  & \textrm{if $0\le u\le \beta$,}\\
 u & \textrm{if $u\ge \beta$.}
 \end{cases}
 \]
 Then $g_{\alpha, \beta}$ is a homeomorphism from $\RR_{\ge 0}$ to $\RR_{\ge \alpha}$ 
 that is the identity on $\RR_{\ge \beta}$.
\erase{
\begin{figure}
 \begin{tikzpicture}[x=.7cm,y=.7cm]
 \draw [-stealth](1,0) -- (1,6);
  \draw [-stealth](0,1) -- (6,1);
  \draw[very thick] (1,2.3) -- (3,3);
    \draw[very thick]  (3,3) -- (6,6);
        \draw[dashed]  (3,3) -- (3,1);
         \draw[dashed]  (3,3) -- (1,3);
\node at (1,3) [left] {$\beta$};
\node at (3.3,.6) [left] {$\beta$};
\node at (1,2.3) [left] {$\alpha$};
\node at (6,1) [right] {$u$};
\node at (1,6) [above] {$g_{\alpha, \beta}(u)$};
  \end{tikzpicture}
 \caption{Function $g_{\alpha, \beta}$}\label{fig:gaabb}
 \end{figure}
 }
 We choose $\vareep$ so small that the point $(\tilx, m)=(2 \vareep \ss (m), m) $ is 
 in the chart of the coordinate system $(\tilxi, \mu)$ for all $m\in [a, b]$.
Let $U$ denote the chart.
 We then define a displacement 
\begin{equation}\label{eq:displamentofbatashC}
 \betash C \times \intvalp \to  \betash C
\end{equation}
 of $\betash C$ in  $\betash C$ by
  \[
 (Q\sprime, t)\mapsto 
  \begin{cases}
  (g_{t \ss(m), 2t \ss(m)}(\tilx), m) & \textrm{if $Q\sprime=(\tilx, m) \in U\cap \betash C$}, \\
  Q\sprime & \textrm{otherwise.}
  \end{cases}
 \]
Note that this map is continuous because we have $g_{t \ss(m), 2t \ss(m)}(\tilx)=\tilx$ for $\tilx\ge 2t \ss(m)$. 
 Let $(\betash C)_{t}$ be the closed subset of $\betash C$ obtained by
 removing 
 \[
 \set{(\tilx, m)}{a\le m\le b, \;\; 0\le \tilx < t \ss(m)}
 \]
 from $\betash C$.
Then the displacement~\eqref{eq:displamentofbatashC} is a family of maps that shrink
$\betash C$ to $(\betash C)_{t}\subset \betash C$ homeomorphically.
 Note that the displacement~\eqref{eq:displamentofbatashC} 
 preserves the subspaces $\betash C\cap \betash\ell_a(\CC)$ and $\betash C\cap \betash\ell_b(\CC)$.
Putting 
 \[
 \pants(C)_{t}:=\phi\inv ((\betash C)_{t})
 \]
 with the orientation $\sigma_C$, 
we obtain a displacement 
$\set{ \pants(C)_{t}}{t\in \intvalp}$
of $\pants(C)$ in $X$.
We have
\[
\bdr \pants(C)_t=(S_{C, P})_t \;\;\sqcup  \bigsqcup_{P\sprime\in \Vertexes (C)\setminus \{P\}} S_{C, P\sprime}.
\]
Therefore we have a displacement 
\[
\stdvan(C)_{t}\;\;:=\;\; \pants(C)_{t} \;\cup\; \left( (H_{C,  P})_t \;\;\sqcup\;\; \bigsqcup_{P\sprime\in \Vertexes (C)\setminus \{P\}} H_{C, P\sprime}\right)
\]
of the topological  $2$-cycle $\stdvan(C)_0=\stdvan(C)$ in $X$.
These displacements 
\begin{equation}\label{eq:singleE}
\renewcommand{\arraystretch}{1.4}
\begin{array}{cl}
 \set{(\betash C)_{t}}{t\in \intvalp} & \textrm{ of $\betash C$ in $\betash C$,}\\
 \set{\pants(C)_{t}}{t\in \intvalp} & \textrm{ of $\pants(C)$ in $\pants(C)$, and }\\
\set{\stdvan(C)_{t}}{t\in \intvalp}& \textrm{ of $\stdvan(C)$ in $X$ }
\end{array}
\end{equation}
are also called 
the \emph{$E$-displacements} 
associated with the section $s\colon E_P\to \NNN_P$.
%
\par
We can easily extend the definition of $E$-displacements 
to the case where we are given sections at several vertexes.
Note that the displacements in~\eqref{eq:singleE} are stationary outside  small neighborhoods 
of $J_{C, P}$ in $\betash C$,
of $S_{C, P}$ in $\pants(C)$,
and of $H_{C, P}$ in $\stdvan(C)$, respectively.
By choosing sufficiently small $\vareep$, 
we can make these neighborhoods  arbitrarily small.
Therefore, if we are given  sections 
\[
s_i\colon E_{P_i} \to \NNN_{P_i} \;\; (i=1, \dots, k)
\]
admissible with respect to $C$ 
 for distinct vertexes 
$P_1, \dots, P_k$ of $C$, 
we can glue the displacements associated with these $s_i$ together. 
For example, 
let $\UUU(P_i)$ be a neighborhood of $H_{C, P}$ in $\stdvan(C)$
such  that the $E$-displacement $d_{s_i}\colon \stdvan(C)\times \intvalp\to X$ associated with $s_i$ is stationary outside of  $\UUU(P_i)$.
Then we can define  a displacement
\[
d:=d_{s_1, \dots, s_k}\colon \stdvan (C) \times \intvalp  \to X 
\]
by the following: 
\[
d(Q, t):=\begin{cases}
d_{s_i} (Q, t) &\textrm{if $Q\in \UUU(P_i)$ for some $i$, }\\
Q&\textrm{if $Q\notin \UUU(P_i)$ for any $i$.}
\end{cases}
\]
\begin{definition}
These displacements  of $\betash C$ in $\betash C$,
of $\pants(C)$ in $\pants(C)$, and 
of $\stdvan(C)$ in $X$, 
are called the \emph{$E$-displacements}
associated with the continuous sections $s_i\colon E_{P_i} \to \NNN_{P_i}$.
\end{definition}
By construction,
the $E$-displacements associated with these sections $s_i\colon E_{P_i} \to \NNN_{P_i}$ ($i=1, \dots, k$)
satisfy the following:
\begin{equation}\label{eq:disjointafterE}
(\betash C)_{\vareep} \cap E_{P_i}=\emptyset\;\textrm{ and hence }\;
\pants(C)_{\vareep} \cap D_{P_i}=\emptyset
\quad\textrm{for  $i=1, \dots, k$.}
\end{equation}
\section{Proof of Theorem~\ref{thm:intnumbs} (1), (2), (3)}\label{sec:proofofThmB}
We  prove assertions (1), (2), (3) of Theorem~\ref{thm:intnumbs}
for the standard orientations $\sigma_C$  fixed in Section~\ref{subsec:StandardOrientations}.
\subsection{Proof of assertion (1)}\label{subsec:proof(1)}
The case where $P\notin \Vertexes (C)$  is obvious.
 Suppose that  $P\in \Vertexes (C)$.
 As in Section~\ref{subsec:Edisp}, 
 we choose a  section $s\colon E_P\to \NNN_P$ 
 of a tubular neighborhood $\NNN_P\to E_P$ 
 admissible  with respect to $C$, 
 and construct the $E$-displacement $ \set{\stdvan(C)_t}{t\in \intvalp}$
 associated with $s$.
Note that  $\pants(C)_{\vareep}$ is disjoint from $D_P$.
(See~\eqref{eq:disjointafterE}.)
 Recall that $(E_{P})_\vareep$ and $E_{P}$  intersect only at one point $Q_0$,
 and the local intersection number at $Q_0$ is $-1$.
 Since $Q_0\notin J_{C, P}$,
 we see that $(D_{P})_\vareep$ and  $D_{P}$  intersect only at the two points in $\phi\inv(Q_0)$
 and that the local intersection number at each point  is $-1$.
Assuming $\vareep$ to be small enough, 
we have that one of  the two points in $\phi\inv(Q_0)$ is on $(H_{C,  P})_{\vareep}$,
 whereas the other does not belong to $(H_{C, P})_{\vareep}$.
 Hence $\stdvan(C)_{ \vareep}$ and $D_{P}$ intersect only at one point, 
 at which  the local intersection number is $-1$.
 Thus  assertion (1) is proved. 
 \qed
 \subsection{Proof of assertion (2)}\label{subsec:proof(2)}
 Obvious.  \qed
\subsection{Proof of assertion (3)}\label{subsec:proof(3)}
Suppose that $C\cap C\sprime$ consists of a single point $P$.
 Let $(\xi, \eta)$, $(\tilxi, \mu)$,  $(x, y)$, $(\tilx, m)$ 
 be good local coordinates on $\AtC$, $\YC$, $\AtR$, $\YR$, respectively, 
 given in Section~\ref{subsec:goodsystem}.
 Without loss of generality,
 we can assume that the location of $C$ is Case (1) 
 and the location of $C\sprime$ is Case (3) in~\eqref{eq:Ccases}.
 Note that we have
 \[
 J_{C, P}=J_{C\sprime, P}\;\;\textrm{on $E_P$}, \quad  S_{C, P}=S_{C\sprime, P}\;\;\textrm{on $D_P$}.
 \]
By Corollary~\ref{cor:coherent},
 we see that 
 $H_{C, P}$ and $H_{C\sprime, P}$
 are the closures of \emph{distinct} connected components 
 of $D_P\setminus S_{C, P}=D_P\setminus S_{C\sprime, P}$.
 \par
 As in Section~\ref{subsec:Edisp}, 
 we choose a section $s\colon E_P\to \NNN_P$ 
 of  a tubular neighborhood $\NNN_P\to E_P$ 
 admissible with respect to $C$, 
 and construct the $E$-displacement $\stdvan(C)_{\vareep}$  
 associated with $s$.
We illustrate the $E$-displacements of  $\betash C$ and $E_P$ in $\YC$ in Figure~\ref{fig:displacement3},
where $Q_0$ is the zero of the section $s$.
  Since $\betash C\sprime$ and $(\betash C)_{\vareep}$ are disjoint,
 it follows that $\pants(C\sprime)$ and $\pants(C)_{\vareep}$ are disjoint.
  Since $E_P$ and $(\betash C)_{\vareep}$ are disjoint,
we see that $H_{C\sprime, P}$ and $\pants(C)_{\vareep}$ are disjoint.
  Since $\betash C\sprime$ and $(E_{P})_{ \vareep}$ are disjoint,
 it follows that $\pants(C\sprime)$ and  $(H_{C, P})_{\vareep}$ are disjoint.
Note that $D_{P}$ and $(D_{P})_{\vareep}$  intersect only at the two  points in $\phi\inv(Q_0)$.
Since  $H_{C, P}$ and $H_{C\sprime,  P}$ are distinct hemispheres of $D_P$ given by $S_{C, P}=S_{C\sprime, P}$,
and $\vareep$ is sufficiently small, 
neither of the two  points of $\phi\inv(Q_0)$
is on $ H_{C\sprime, P}\cap (H_{C, P})_{ \vareep}$.
Therefore $\stdvan(C\sprime)$ and $(\stdvan(C))_{\vareep}$ are disjoint.
%
\begin{figure}
\begin{tikzpicture}[x=.7cm, y=.7cm]
\fill[lightgray] (-3, 2.5) rectangle (0,5);
\fill[lightgray] (1, 2.5) rectangle (4,5);
\draw[thick] (0, 1)--(0, 9);
\draw[thick] (1, 1)[rounded corners] -- (1, 6)[rounded corners]--(-1,8)--  (-1,9);
\draw[thick] (-3, 5)--(0, 5);
\draw[thick] (1, 5)--(4, 5);
\draw[thick] (-3, 2.5)--(0, 2.5);
\draw[thick] (1, 2.5)--(4, 2.5);
\node at (-0.5, 0.5) [right] {$E_P$};
\node at (0.5, 0.5) [right] {$(E_P)_{\varee\sprime}$};
\node at (0.3,7) [right]{$Q_0$};
\fill (0,7) circle (2pt);
\node at (-2,3.75) [right]{$\betash C\sprime$};
\node at (1.5,3.75) [right]{$(\betash C)_{\varee\sprime}$};
\end{tikzpicture}
\caption{$E$-displacement for the proof of assertion (3)}\label{fig:displacement3}
\end{figure}
 \qed
\section{Proof of Theorem~\ref{thm:intnumbs} (4)}\label{sec:Proof4}
We  prove assertion (4) of Theorem~\ref{thm:intnumbs}
for the standard orientations $\sigma_C$  fixed in Section~\ref{subsec:StandardOrientations}.
We first prove  that
the intersection number $\intf{[\stdvan(C)], [\stdvan(\Cp)]}$
for distinct bounded chambers $C$ and $\Cp$
sharing a common edge
does not depend on  $\Arr$. See~Lemma~\ref{lem:same}.
We then compare 
the intersection number $\intf{[\stdvan(C)], [\stdvan(\Cp)]}$ 
with the intersection number $\intf{[\stdvan\spm(C)], [\stdvan\spm(\Cp)]}\spm$
that is calculated by replacing the defining polynomial $f$ of $\Arr$ with $-f$.
Using~\eqref{eq:orientationreversing}, 
we obtain a proof of assertion~(4) of Theorem~\ref{thm:intnumbs}.
\subsection{A preliminary lemma}\label{subsec:preliminary}
We consider two nodal real line arrangements $\Arr_1$ and $\Arr_2$.
Let $j$ be in $\{1, 2\}$.
Let $\Arr_j$ be a nodal real line arrangement  with 
a defining polynomial $f_j$.
We consider the morphisms
\[
X_j\maprightsp{\rho_j} W_j \maprightsp{\pi_j} \AtC.
\]
Here $W_j$ is defined in $\CC\times \AtC$
by 
\[
\omega_j^2=f_j,
\]
where $\omega_j$ is an affine coordinate of $\CC$, 
and $\pi_j$ denotes the double covering
$(\omega_j, Q)\mapsto Q$.
The morphism $\rho_j\colon X_j\to W_j$ is the minimal desingularization.
Then, for a bounded chamber $C$ of $\Arr_j$,
we have a vanishing cycle $\stdvan(C)$ in $X_j$
with the standard orientation determined by $\sigmaA$ and $f_j$.
We denote by
\[
\intfnull_j\colon H_2(X_j; \ZZ)\times H_2(X_j; \ZZ)\to \ZZ
\]
the intersection form on $X_j$.
\begin{lemma}\label{lem:same}
Let $C_j$ and $\Cp_j$ be distinct bounded chambers
of $\Arr_j$ that share a common edge.
Then we have
\begin{equation}\label{eq:same}
\intf{[\stdvan(C_1)], [\stdvan(\Cp_1)]}_1=\intf{[\stdvan(C_2)], [\stdvan(\Cp_2)]}_2.
\end{equation}
\end{lemma}
\begin{proof}
The idea 
is to construct a diffeomorphism 
from an  open neighborhood $\UX{1}$ of $\stdvan(C_1)\cap \stdvan(\Cp_1)$
in $X_1$ to 
an open neighborhood $\UX{2}$ of $\stdvan(C_2)\cap \stdvan(\Cp_2)$
in $X_2$ 
that induces orientation-preserving isomorphisms 
\begin{equation}\label{eq:inducestdvans}
\UX{1} \cap \stdvan(C_1)\;\cong \;\UX{2} \cap \stdvan(C_2)
\quand 
\UX{1} \cap \stdvan(\Cp_1)\;\cong \;\UX{2} \cap \stdvan(\Cp_2).
\end{equation}
Since the intersection number 
of topological cycles is determined by 
the geometric situation  in a neighborhood of the 
set-theoretic intersection of the topological cycles,
this diffeomorphism   proves~\eqref{eq:same}.
In fact, we will construct a  biholomorphic isomorphism $\psi_X\colon \UX{1}\isom\UX{2}$
explicitly.
\par
Since  $C_j$ and $\Cp_j$ are adjacent,
the signs of $f_j$ on $C_j$ and on $\Cp_j$ are opposite.
Interchanging $C_j$ and $\Cp_j$ if necessary,
we can assume without loss of generality that
\[
f_j(C_j)\subset \RR_{\ge 0}, \quad 
f_j(\Cp_j)\subset \RR_{\le 0}.
\]
Let $e_j:=P_j \Pp_j$ be the common edge of $C_j$ and $\Cp_j$, 
where we have 
\[
\Vertexes  (C_j)\cap \Vertexes  (\Cp_j)=\{P_j, \Pp_j\}.
\]
Let $\ellej(\RR)\in \Arr_j$ be the line defining the edge $P_j\Pp_j$.
Let $(\xi_j, \eta_j)$ be good affine coordinates of $\AtC$
 with $(x_j, y_j)=(\Real \xi_j, \Real \eta_j)$
regarded as affine coordinates of $\AtR$
such that $\ellej(\RR)$ is defined by $y_j=0$, 
and such that $\{P_j, \Pp_j\}$ is  $\{(\pm 1, 0)\}$.
See~Section~\ref{subsec:goodcoordinates}.
Note that, if $(\xi_j, \eta_j)$ satisfies these properties, then so does $(-\xi_j, -\eta_j)$.
Replacing  $(\xi_j, \eta_j)$ with $(-\xi_j, -\eta_j)$ if necessary,
we can assume that 
\[
C_j\subset \{y_j\ge 0\},
\quad
\Cp_j\subset \{y_j\le 0\}.
\]
Interchanging  $P_j$ and  $\Pp_j$ if necessary,
we can further assume that
\[
P_j=(-1, 0), \quad \Pp_j=(1, 0)
\]
in terms of $(x_j, y_j)$.
Let $\ellPj (\RR)$  (resp. $\ellPpj (\RR)$) be the line in $\Arr_j$ other than $\ellej(\RR)$ 
that passes through $P_j$ (resp. $\Pp_j$).
Then there exist real numbers $g_j, \gp_j$ 
such that
\[
\ellPj (\RR)\;\colon\; x_j=-1 +g_j y_j,
\quad 
\ellPpj (\RR)\;\colon\; x_j=1 +\gp_j  y_j.
\]
See Figure~\ref{fig:CjCpj}.
\begin{figure}
\begin{tikzpicture}[x=19mm, y=19mm]
\draw[thick](-1.4,0)--(1.4, 0);
\draw[thick](-.8,.9)--(-1.2,-.9);
\draw[thick](.8,.9)--(1.2,-.9);
\fill (-1,0) circle (2pt);
\fill (1,0) circle (2pt);
\node at (0, .8) [below] {$C_j$};
\node at (0,-.9) [above] {$\Cp_j$};
\node at  (1.25,-.9)[below] {$\ell_{\Pp j}(\RR)$};
\node at  (-1.2,-.9)[below] {$\ell_{P j}(\RR)$};
\node at  (1.4, 0)[right]  {$\ell_{e j}(\RR)$};
\node at  (-1.05, .2)[left]  {$P_j$};
\node at  (1.05, .2)[right]  {$\Pp_j$};
\end{tikzpicture}
\caption{$C_j$ and $\Cp_j$}\label{fig:CjCpj}
\end{figure}
\par
We prepare another copy of $\AtC$
with  good affine coordinates
$(\xi_0, \eta_0)$, and  
we regard $(x_0, y_0)=(\Real \xi_0, \Real \eta_0)$
as affine coordinates of $\AtR$.
We put
\[
P_0:=(-1, 0), \quad \Pp_0:=(1, 0), \quad e_0:=P_0 \Pp_0, 
\]
and define three lines $\ellez (\RR), \ellPz (\RR), \ellPpz (\RR)$ by 
\[
\ellez (\RR)\;\colon\; y_0=0,
\quad 
\ellPz (\RR)\;\colon\; x_0=-1,
\quad 
\ellPpz (\RR)\;\colon\; x_0=1.
\]
We then put
\[
\begin{split}
C_0&:=\set{(x_0, y_0)}{-1\le x_0\le 1, \; 0\le y_0},\\
\Cp_0&:=\set{(x_0, y_0)}{-1\le x_0\le 1, \; 0\ge y_0}.
\end{split}
\]
Finally, we put
\[
f_0:=y_0 (1-x_0) (1+x_0),
\]
which is a defining polynomial of the nodal real line arrangement $\Arr_0$ 
consisting of three lines $\ellez (\RR), \ellPz (\RR), \ellPpz (\RR)$.
Note  that we have 
\[
f_0(C_0)\subset \RR_{\ge 0}, \quad f_0(\Cp_0)\subset \RR_{\le 0}.
\]
\par
We define a quadratic map $\psi_j\colon \AtC\to \AtC$ by
\[
(\xi_0, \eta_0)\;\; \mapsto\;\;  (\xi_j, \eta_j)
=\left(\xi_0+\left(\frac{1-\xi_0}{2} g_j +\frac{1+\xi_0}{2} \gp_j\right)\eta_0, \; \eta_0\right),
\]
which maps $P_0$ to $P_j$ and $\Pp_0$ to $\Pp_j$.
Let $U_0$ be a sufficiently small open neighborhood of $e_0$ in $\AtC$.
Then there exists an open neighborhood $U_j$ of 
 $e_j$ in $\AtC$ 
 such that $\psi_j$ induces a biholomorphic isomorphism
 \[
 \psi_{j U}\colon U_0 \isom  U_j.
 \]
 The inverse  of $ \psi_{j U}$  is given by
 \[
 (\xi_j, \eta_j)\;\; \mapsto\;\;  (\xi_0, \eta_0)
 =\left(\frac{2 \xi_j -(g_j+\gp_j) \eta_j}{2-(g_j-\gp_j)\eta_j}, \; \eta_j\right).
 \]
We can easily check that $ \psi_{j U}$ maps 
$U_0\cap \AtR$ to $U_j\cap \AtR$ isomorphically and orientation-preservingly, 
and that it also maps isomorphically
\[
\renewcommand{\arraystretch}{1.2}
\begin{array}{lcl}
 U_0 \cap C_0 &\textrm{to} &  U_j \cap C_j, \\
 U_0 \cap \Cp_0 &\textrm{to} &  U_j \cap \Cp_j, \\
 U_0 \cap \ellez(\CC) &\textrm{to} &  U_j \cap \ellej(\CC), \\
  U_0 \cap \ellPz(\CC) &\textrm{to} &  U_j \cap \ellPj(\CC), \\
  U_0 \cap \ellPpz(\CC) &\textrm{to} &  U_j \cap \ellPpj(\CC).
 \end{array}
\]
We make $U_0$ so small that
$U_j$ does not intersect any element of the complexified line arrangement $\Arr_{j \CC}$ other than
$\ellej(\CC), \ellPj(\CC),  \ellPpj(\CC)$.
Then
\[
u_j:=f_0/ \psi_{j U}^* f_j
\]
is a holomorphic function on $U_0$ 
that does \emph{not} have any zero.
Moreover we have $u_j (Q)\in \RR_{>0}$ for any $Q\in C_0\cup \Cp_0$.
We can assume that $U_0$ is simply-connected.
Then we can define a holomorphic function 
$v_j$ on $U_0$ such that $v_j^2=u_j$ (in particular $v_j$ does not have any zero)
and that 
\begin{equation}\label{eq:vjQ}
v_j (Q)\in \RR_{>0}\;\;\textrm{for any}\;\;  Q\in C_0\cup \Cp_0.
\end{equation}
Let $W_0$  be  defined in $\CC\times \AtC$
by $\omega_0^2=f_0$, 
where $\omega_0$ is an affine  coordinate of $\CC$, 
and let $\pi_0\colon W_0\to \AtC$ be  the double covering
$(\omega_0, Q_0)\mapsto Q_0$.
We put
\[
\UW{0}:=\pi_0\inv (U_0),
\quad 
\UW{j}:=\pi_j\inv (U_j).
\]
We define an isomorphism
\[
\psi_{j W}\colon \UW{0} \isom \UW{j}
\]
that covers the isomorphism $\psi_{j U}\colon  U_0\isom  U_j$ by 
\begin{equation}\label{eq:psijW}
(\omega_0, Q_0)\mapsto (\omega_j, Q_j)=(\omega_0/v_j(Q_0), \psi_{j U} (Q_0)).
\end{equation}
We put 
\[
\psi_W:=\psi_{2 W} \circ \psi_{1 W}\inv \;\colon\; \UW{1}\isom \UW{2},
\]
which induces 
isomorphisms 
from  $\UW{1} \cap \pi_1\inv(C_1)$ to 
$\UW{2} \cap \pi_2\inv(C_2)$ 
and 
from  $\UW{1} \cap \pi_1\inv(\Cp_1)$ to 
$\UW{2} \cap \pi_2\inv(\Cp_2)$.
Recall from Definition~\ref{def:sigmaC} that the positive-sheet and the negative-sheet 
are defined by the sign of $\omega\in \RR$ or $\omega/\thei \in \RR$.
Therefore
it follows from~\eqref{eq:vjQ} and~\eqref{eq:psijW} that 
$\psi_W$  maps
the positive-sheet of $\UW{1}\cap  \pi_1\inv(C_1\spcirc)$
to   the positive-sheet of $\UW{2}\cap \pi_2\inv(C_2\spcirc)$,
and 
the positive-sheet of $\UW{1}\cap  \pi_1\inv(C\sp{\prime\circ}_1)$
to   the positive-sheet of $\UW{2}\cap \pi_2\inv(C\sp{\prime\circ}_2)$.
Hence the isomorphism
by $\psi_W$ from $\UW{1} \cap \pi_1\inv(C_1)$ to 
$\UW{2} \cap \pi_2\inv(C_2)$ is orientation-preserving,
where orientations are the standard orientations defined by $\sigmaA$ and $f_j$.
We put
\[
\UX{j}:=\rho_j\inv (\UW{j}).
\]
Since the minimal resolution of an ordinary node of an algebraic  surface is unique, 
the isomorphism $\psi_W$ induces an isomorphism 
\[
\psi_X \;\colon\; \UX{1}\isom \UX{2}
\]
that maps $D_{P_1}$ to  $D_{P_2}$ and $D_{\Pp_1}$ to  $D_{\Pp_2}$.
Since $\psi_X$ induces orientation-preserving isomorphisms
\[
U_{X,1}\cap \pants(C_1)\;\cong\; U_{X,2}\cap\pants(C_2),\quad 
U_{X,1}\cap\pants(\Cp_1)\;\cong\;U_{X,2}\cap\pants(\Cp_2),
\]
it follows that $\psi_X$ induces isomorphisms
of the capping hemispheres 
\[
\renewcommand{\arraystretch}{1.2}
\begin{array}{lllllll}
H_{C_1, P_1} &\isom  &  H_{C_2, P_2}, &  & H_{C_1, \Pp_1} &\isom  &  H_{C_2, \Pp_2},  \\
H_{\Cp_1, P_1} &\isom  &  H_{\Cp_2, P_2}, & & H_{\Cp_1, \Pp_1} &\isom  &  H_{\Cp_2, \Pp_2}.  \\
 \end{array}
\]
Therefore $\psi_X$ induces 
orientation-preserving isomorphisms~\eqref{eq:inducestdvans}.
\par
Note that $\stdvan(C_j)\cap \stdvan(\Cp_j)$ is contained in  $\UX{j}$.
The intersection number  of 
$\stdvan(C_1)$ and  $\stdvan(\Cp_1)$ 
is calculated by making a small displacement $\stdvan(\Cp_1)_{\varee}$ 
of $\stdvan(\Cp_1)$
that is stationary  outside $\UX{1}$
so that $\stdvan(C_1)$ and $\stdvan(\Cp_1)_{\varee}$
intersect transversely.
Transplanting  this procedure 
into $\UX{2}$ by $\psi_X$,
we see that $\stdvan(C_2)$ and  a  displacement $\stdvan(\Cp_2)_{\varee}$
of $\stdvan(\Cp_2)$
intersect in the same way as $\stdvan(C_1)$ and $\stdvan(\Cp_1)_{\varee}$.
Therefore we obtain~\eqref{eq:same}.
\end{proof}
\subsection{Proof of assertion (4) of Theorem~\ref{thm:intnumbs}}\label{subsec:changingsigns}
%
%
Interchanging $C$ and $\Cp$ if necessary,
we can assume that
\begin{equation}\label{eq:fCCp}
f(C)\subset \RR_{\ge 0},
\quad
f(\Cp)\subset \RR_{\le 0}.
\end{equation}
We construct morphisms~\eqref{eq:rhopi} with $f$ replaced by $-f$, 
and denote them  by 
\begin{equation}\label{eq:rhopim}
X\spm \maprightsp{\rho\spm } W\spm  \maprightsp{\pi\spm} \AtC,
\end{equation}
where $W\spm$ is defined in $\CC\times \AtC$ by
\[
\omega^2=-f
\]
with $\pi\spm (\omega, Q)=Q$,
and $\rho\spm\colon X\spm \to W\spm$ is the minimal desingularization.
For a bounded chamber $C\spprime$,
let $\stdvan\spm(C\spprime)$ denote the vanishing cycle over $C\spprime$ in $X\spm$ 
equipped with the standard orientation $\sigma\spm_{C\spprime}$ defined by $\sigmaA$ and $-f$.
Let 
\[
\intfnull\spm\colon H_2(X\spm; \ZZ)\times H_2(X\spm; \ZZ) \to \ZZ
\]
denote the intersection form on $X\spm$.
By Lemma~\ref{lem:same},
we have 
\begin{equation}\label{eq:sameminus}
\intf{[\stdvan(C)], [\stdvan(C\sprime)]}=
\intf{[\stdvan\spm(C)], [\stdvan\spm(C\sprime)]}\spm.
\end{equation}
Let $\Phi_W\colon W\isom  W\spm$ be the isomorphism defined by
\[
(\omega, Q)\mapsto (\thei \omega, Q).
\]
Note that we have 
\[
\omega\in \RR_{\ge 0} \Longrightarrow \thei \omega\in \thei\RR_{\ge 0},
\quad
\omega\in \thei \RR_{\ge 0} \Longrightarrow \thei \omega\in -\RR_{\ge 0}.
\]
Therefore, by~\eqref{eq:fCCp}, we see that  
$\Phi_W$ maps 
the  positive-sheet of $\pi\inv (C\spcirc)$ 
to the positive-sheet of $({\pi\spm})\inv (C\spcirc)$,
whereas it  maps 
the  positive-sheet of $\pi\inv (C\sp{\prime \circ})$ 
to the \emph{negative}-sheet of $({\pi\spm})\inv (C\sp{\prime \circ})$.
Therefore the isomorphism $\Phi_X\colon X\isom X\spm$
induced by $\Phi_W$ induces 
an 
isomorphism  
$\stdvan(C, \sigma_C)\cong \stdvan\spm(C, \sigma\spm_C)$, 
 whereas it induces 
 an isomorphism $\stdvan(C\sprime, \sigma_{C\sprime})\cong \stdvan\spm(C\sprime,-\sigma\spm_{C\sprime})$.
 Hence the homomorphism $\Phi_{X*}\colon H_2(X; \ZZ)\isom H_2(X\spm; \ZZ)$
 satisfies
%
\begin{align}
 \Phi_{X*}([\stdvan(C)])&=[\stdvan\spm(C)], \label{eq:PhiX1}\\
 \quad
 \Phi_{X*}([\stdvan(\Cp)])&=\sum_{P\spprime\in \Vertexes  (\Cp)}[D_{P\spprime}]-[\stdvan\spm(\Cp)],  \label{eq:PhiX2}
 \end{align}
%
 where~\eqref{eq:PhiX2} is derived from~\eqref{eq:orientationreversing}.
 Since $\Phi_{X*}$ preserves the intersection form,
 we  have
 \begin{equation}\label{eq:PhiXintnumbs}
 \intf{[\stdvan(C)], [\stdvan(\Cp)]}=
 \intf{\Phi_{X*}([\stdvan(C)]), \Phi_{X*}([\stdvan(\Cp)])}\spm.
  \end{equation}
Comparing~\eqref{eq:sameminus},~\eqref{eq:PhiX1},~\eqref{eq:PhiX2},~\eqref{eq:PhiXintnumbs},
we obtain
\[
2\, \intf{[\stdvan(C)], [\stdvan(\Cp)]}= 
\sum_{P\spprime\in \Vertexes  (\Cp)}\intf{[\stdvan\spm(C)], [D_{P\spprime}]}\spm=-2,
\]
where we use assertion (1) of Theorem~\ref{thm:intnumbs} 
and $|\Vertexes  (C)\cap \Vertexes  (\Cp)|=2$  in the second equality.
Thus Theorem~\ref{thm:intnumbs} (4) is proved.
\qed
\section{Proof of Theorem~\ref{thm:intnumbs} (5)}\label{sec:Proof5}
We  prove assertion (5) of Theorem~\ref{thm:intnumbs}
for the standard orientations $\sigma_C$  fixed in Section~\ref{subsec:StandardOrientations}.
\subsection{Target displacements}\label{subsec:targetProof5}
We construct an $E$-displacement 
\begin{equation}\label{eq:E}
\stdvan(C)_{\vareep}\;\;=\;\;\pants(C)_{\vareep} 
  \cup 
  \bigsqcup_{P\in \Vertexes (C)} (H_{C, P})_{\vareep}
 \end{equation}
of $\stdvan(C)$ in $X$ associated with sections 
$s_P\colon E_P\to \NNN_P$ for all $P\in \Vertexes (C)$,
and 
a $C$-displacement
\begin{equation}\label{eq:C}
\stdvan(C)_{\varee}\;\;=\;\;
\pants(C)_{\varee} \cup  \bigsqcup_{P\in \Vertexes (C)} H_{C, P} 
 \end{equation}
of $\stdvan(C)$ in $X$ associated with a continuous  function $\delta\colon C\to T(\AtR)$.
By the definition of $E$-displacements~(see~\eqref{eq:disjointafterE}),
for any $P\in \Vertexes  (C)$, 
we have  that $\pants(C)_{\vareep}$ and $H_{C, P}$ are disjoint,
and that $(H_{C, P})_{\vareep}$ and $H_{C, P}$ intersect only at one point, at which 
 the local intersection number is $-1$.
By choosing a sufficiently small $\vareep$,
we have that $(H_{C, P})_{\vareep}$ and $H_{C, \Pp}$ are disjoint for $P\ne \Pp$.
We  construct these displacements~\eqref{eq:E} and~\eqref{eq:C}  in such a way that
they have the following additional properties.
\begin{itemize}
\item[(EC1)]
The two spaces 
$\pants(C)_{\vareep} $ and $\pants(C)_{\varee} $  intersect only at two  points,
each  intersection point is in the interior of $\pants(C)_{\vareep} $ and of $\pants(C)_{\varee} $, 
and the local intersection number is $-1$ at each  intersection point.
\item[(EC2)]  For any $P\in \Vertexes  (C)$, we have that 
$(H_{C, P})_{\vareep}$ and $\pants(C)_{\varee} $ intersect only at one point,
that the intersection point is in the interior of $(H_{C, P})_{\vareep}$ and of $\pants(C)_{\varee} $, 
and that the local intersection number is $1$. 
\end{itemize}
Then we have
\[
\begin{split}
\intf{[\stdvan(C)], [\stdvan(C)]}&=\intf{[\stdvan(C)_{\vareep}], [\stdvan(C)_{\varee}]}=\\
&=\sum_{P\in \Vertexes (C)} (-1 )+ \sum_{P\in \Vertexes (C)}  1 +(-2)=-2,
\end{split}
\]
where the first sum comes from the intersections of  $(H_{C, P})_{\vareep}$ and $H_{C, P}$,
the second sum comes from 
the intersections of  $(H_{C, P})_{\vareep}$ and $\pants(C)_{\varee} $,
and the third  term $(-2)$ 
comes from the intersection of $\pants(C)_{\vareep} $ and $\pants(C)_{\varee} $.
Thus assertion (5) of Theorem~\ref{thm:intnumbs} will be proved.
\subsection{Construction of the \texorpdfstring{$E$}{E}-displacement}\label{subsec:proof5E}
For each $P\in \Vertexes  (C)$,
we fix good local coordinates $(\xi_P, \eta_P)$, $(x_P, y_P)$, $(\tilxi_P, \mu_P)$, and $(\tilx_P, m_P)$ 
on $\AtC$, $\AtR$, $\YC$,  and  $\YR$,  
 given in Section~\ref{subsec:goodsystem}.
 We can assume, without loss of generality, 
 that  $C$ is located as Case~(1) in~\eqref{eq:Ccases}
so that there exist real numbers  $a_P, b_P$ with $a_P<b_P$ such that 
 $ \betash C$ is give  by 
$\tilx_P\ge 0$ and  $ a_P\le m_P\le b_P$  
in the chart of $(\tilx_P, m_P)$ on $\YR$.
To make the computation easy,
we choose $(\xi_P, \eta_P)$
in such a way that  we have
\begin{equation}\label{eq:aPbPfixed}
a_P=-1, \quad b_P=1.
\end{equation}
See Remark~\ref{rem:arbitrarypair}.
We choose a  sufficiently small real number $c$,
and choose a section $s_P\colon E_P\to \NNN_{P}$ of
a tubular neighborhood $ \NNN_{P}\to E_P$   
such that, in a neighborhood of  $J_{C, P}$ in $E_P$,
the section $s_P$ is written as
\[
(0, \mu)\mapsto (c, \mu)
\]
in terms of $(\tilxi_P, \mu_P)$.
Then $s_P$ is admissible with respect to $C$.
See Example~\ref{example:constantsection}.
Now we construct the $E$-displacement~\eqref{eq:E}
of $\stdvan(C)$ in $X$
associated with all these sections $s_P$.
\par
We denote by $\varGammaPtriangle $ 
the image by $\beta$ of the region
\[
\set{(\tilx_P, m_P)}{0\le \tilx_P\le c\vareep, \; -1\le m_P\le 1}
\]
in $\betash C$.
In terms of the coordinates $(x_P, y_P)$ of $\AtR$, we have  
\begin{equation}\label{eq:varGamma}
\varGammaPtriangle =\set{(x_P, y_P)}{0\le x_P\le c\vareep, \; -x_P\le y_P\le x_P} \;\;\subset\;\; C.
\end{equation}
\subsection{Construction of the \texorpdfstring{$C$}{C}-displacement}\label{subsec:proof5C}
As in~Section~\ref{subsec:Cdislpcements},
we introduce an inner-product on the $\RR$-vector space $T(\AtR)$.
In particular, we have a distance and angles on $\AtR$.
\par
Let $n$ be the size of  $\Vertexes (C)$.
We index the vertexes of $C$ by elements $\nu$ of the cyclic group $\ZZ/n\ZZ$
as
\[
\Vertexes (C)=\set{P_{\nu}}{\nu\in \ZZ/n\ZZ}
\]
in such a way that,
for each $\nu\in \ZZ/n\ZZ$,
the line segment $P_{\nu} P_{\nu+1}$ is an edge of $C$,
and that the cyclic sequence $P_0, P_1, \dots, P_{n-1}, P_n=P_0$
of vertexes 
goes around $C$ in the positive-direction with respect to the fixed orientation
$\sigma_{\AA}$ of $\AtR$.
Let $\ell_{\nu}(\RR)\in \Arr$
be the line defining the edge $P_{\nu} P_{\nu+1}$,
and let $\overrightarrow{P_{\nu}  P_{\nu+1} }\in T(\AtR)$ be the translation such that
\[
P_{\nu} +\overrightarrow {P_{\nu}  P_{\nu+1} }=P_{\nu+1}.
\]
Obviously we have $\overrightarrow {P_{\nu}  P_{\nu+1} }\in \Tl{\ell_{\nu} (\RR)}$.
We define a unit vector $\tau_{\nu} \in \Tl{\ell_{\nu} (\RR)}$ by
\[
\tau_{\nu}:=\overrightarrow {P_{\nu}  P_{\nu+1} }\,/\,|P_{\nu}  P_{\nu+1}|.
\]
Let $o\in T(\AtR)$ be the zero vector.
Then $\tau_{\nu+1}$ is obtained 
by rotating $\tau_{\nu}$ in a positive-direction 
by the angle 
\[
\theta_{\nu}:=\angle \tau_{\nu} o \tau_{\nu+1}=\pi-\angle P_{\nu} P_{\nu+1} P_{\nu+2},
\]
where the positive-direction is 
with respect to the orientation of $T(\AtR)$ induced by 
$\sigma_{\AA}$.
We denote by $\Theta_{\nu}$ the triangle $\tau_{\nu} o \tau_{\nu+1}$  in $T(\AtR)$.
Then the interiors of these $n$ triangles are disjoint, and,
 since the sum of  $\theta_{0}, \dots, \theta_{n-1}$
is equal to $2\pi$,
their union 
\[
\Theta:=\bigcup_{\nu \in \ZZ/n\ZZ} \Theta_{\nu}
\]
contains $o$ in its interior.
The cyclic sequence $\tau_0, \dots, \tau_{n-1}, \tau_{n}=\tau_{0}$ 
of the vertexes of the $n$-gon $\Theta$ goes in a positive-direction around $\Theta$.
See Figure~\ref{fig:CTheta}.
\begin{figure}
\begin{tikzpicture}[x=1cm,y=1cm]
\coordinate (t0) at   (-3,1);
\coordinate (t1) at   (-1.25,-2);
\coordinate (t2) at   (1,-1.25);
\coordinate (t3) at   (1,0.33);
\coordinate (t4) at  ($-1*(t0)+-1*(t1)+-1*(t2)+-1*(t3)$);
\coordinate (p0) at (0,0);
\coordinate (p1) at ($(p0)+(t0)$);
\coordinate (p2) at ($(p1)+(t1)$);
\coordinate (p3) at ($(p2)+(t2)$);
\coordinate (p4) at ($(p3)+(t3)$);
\coordinate (m0) at ($1/2*(p0)+1/2*(p1)$);
\coordinate (m1) at ($1/2*(p1)+1/2*(p2)$);
\coordinate (m2) at ($1/2*(p2)+1/2*(p3)$);
\coordinate (m3) at ($1/2*(p3)+1/2*(p4)$);
\coordinate (m4) at ($1/2*(p4)+1/2*(p0)$);
\draw [-{Stealth[scale=1.5]}] (m4)--(p0)--(m0);
\draw [-{Stealth[scale=1.5]}](m0)--(p1)--(m1);
\draw [-{Stealth[scale=1.5]}] (m1)--(p2)--(m2);
\draw [-{Stealth[scale=1.5]}] (m2)--(p3)--(m3);
\draw [-{Stealth[scale=1.5]}] (m3)--(p4)--(m4);
\node at (p0) [right] {$P_0$};
\node at (p1) [above] {$P_1$};
\node at (p2) [left]  {$P_2$};
\node at (p3) [below] {$P_3$};
\node at (p4) [below]  {$P_4$};
\end{tikzpicture}
\qquad
\begin{tikzpicture}[x=2cm,y=2cm]
\coordinate (O) at (0,0);
\coordinate (t0) at   (-0.9486832980, 0.3162277660);
\coordinate (t1) at   (-0.5299989400, -0.8479983040);
\coordinate (t2) at   (0.6246950477, -0.7808688096);
\coordinate (t3) at  (0.9486832984, 0.3162277661);
\coordinate (t4) at  (0.7612432304, 0.6484664555);
\draw (t0)--(t1)--(t2)--(t3)--(t4)--(t0);
\draw [-{Stealth[scale=1.5]}] (O)--(t0);
\draw [-{Stealth[scale=1.5]}] (O)--(t1);
\draw [-{Stealth[scale=1.5]}] (O)--(t2);
\draw[-{Stealth[scale=1.5]}]  (O)--(t3);
\draw [-{Stealth[scale=1.5]}] (O)--(t4);
\node at (t0) [left] {$\tau_0$};
\node at (t1) [below] {$\tau_1$};
\node at (t2) [below]  {$\tau_2$};
\node at (t3) [right] {$\tau_3$};
\node at (t4) [right]  {$\tau_4$};
\draw (O) [very thin] circle (1);
\end{tikzpicture}
\caption{ $C$ and  $\Theta$ }\label{fig:CTheta}
\end{figure}
\par
We define a continuous function 
\[
\delta\colon C\to \Theta \subset T(\AtR)
\]
as follows.
Let  $M_{\nu}$ be the mid-point of the edge $P_{\nu} P_{\nu+1}$,
and let $M$ be the $n$-gon 
$M_0M_1 \dots M_{n-1}M_0$.
Then we have a  homeomorphism
\[
\delta_M\colon M\isom  \Theta
\]
with the following properties.
\begin{itemize}
\item For each $\nu\in \ZZ/n\ZZ$,
we have $\delta_M(M_{\nu})=\tau_{\nu}$.
Moreover,
the restriction of $\delta_M$ to $M_{\nu} M_{\nu+1}$
is an affine-linear isomorphism of line segments from 
$M_{\nu} M_{\nu+1}$ to
$\tau_{\nu} \tau_{\nu+1}$.
In particular, $\delta_M$ is orientation-preserving.
\item The inverse-image $\delta_M\inv (o)$ consists of a single point $Q_0$
in the interior of $M$,
$\delta_M$ is differentiable at the point $Q_0$,
and $Q_0$ is not a critical point of $\delta_M$.
\end{itemize}
For each $\nu\in \ZZ/n\ZZ$,
there exists a unique affine-linear isomorphism
\[
\delta_{\nu}\colon \triangle M_{\nu} P_{\nu+1} M_{\nu+1} \isom \Theta_{\nu}=\triangle \tau_{\nu} o\tau_{\nu+1}
\]
that maps $M_{\nu}$ to $\tau_{\nu}$, $P_{\nu+1}$ to $o$, and $M_{\nu+1}$ to $\tau_{\nu+1}$,
respectively.
In particular,  the map $\delta_{\nu}$ is orientation-\emph{reversing}.
Then the restrictions to the line segment $M_{\nu} M_{\nu+1}$ of $\delta_M$ 
and  of $\delta_{\nu}$  coincide.
Thus, gluing $\delta_M$ and $\delta_{\nu}$ 
along $M_{\nu} M_{\nu+1}$ for $\nu \in \ZZ/n\ZZ$, 
we obtain
a continuous map 
\[
\delta\sprime\colon C\to \Theta \subset T(\AtR).
\]
Note that  $\delta\sprime$ maps the edge $P_{\nu} P_{\nu+1}$
to the line-segment $o\tau_{\nu}\subset \Tl{\ell_{\nu}(\RR)}$.
Hence $\delta\sprime$ satisfies the $e$-condition for every edge of $C$.
\par
We assume that we have chosen $c\in \RR_{>0}$ and $\vareep$ 
in the construction of the $E$-displacement
so small that,
for each $P_{\nu+1}\in \Vertexes (C)$,
 the triangle $\varGamma_{P_{\nu+1}}$  defined by~\eqref{eq:varGamma} 
 is contained in  
$\triangle M_{\nu} P_{\nu+1} M_{\nu+1}$
and that the edge
\[
\gamma_{\nu+1}:=C\cap \{x_{P_{\nu+1}}=c\vareep\}
\]
of 
$\varGamma_{P_{\nu+1}}$
is disjoint from the edge  $M_{\nu}M_{\nu+1}$ of $\triangle M_{\nu} P_{\nu+1} M_{\nu+1}$.
 We choose a sufficiently small positive real number $\rho$
 so that, for each  $P_{\nu+1}\in \Vertexes (C)$, the closed region
 $U_{P_{\nu+1}, \rho}$ defined by~\eqref{eq:UPrho}
 is contained in 
$\varGamma_{P_{\nu+1}}$
 and 
 disjoint from the edge $\gamma_{\nu+1}\subset \varGamma_{P_{\nu+1}}$.
 (See Figure~\ref{fig:atPnu}.)
 We execute the $q_{\rho}$-modification to $\delta\sprime$
at each $P_{\nu+1}\in \Vertexes (C)$,
and obtain a continuous function 
\[
\delta\colon C\to \Theta \subset T(\AtR)
\]
that satisfies the $e$-condition for every edge of $C$ and 
regular at each vertex of $C$.
See Lemma~\ref{lem:qmodification}.
\begin{figure}
\begin{tikzpicture}[x=12mm, y=9mm]
\draw[thick](-1,-1)--(2.5,2.5);
\draw[thick](-1,1)--(2.8,-2.8);
\draw[thick](2,2)--(2.5,-2.5);
\draw[thick](1.3,1.3)--(1.3,-1.3);
\node at (-.2,0) [left] {$P_{\nu+1}$};
\node at (1.3,0) [right] {$\gamma_{\nu+1}$};
\filldraw [fill=gray, opacity=1] (0.7,-0.7)  arc (-45:45:0.989943) --(0,0)--cycle;
\filldraw [fill=gray, opacity=.3] (1.3,-1.3)--(1.3,1.3)--(0,0)--cycle;
\fill (0,0) circle (2pt);
\fill (2.5,-2.5) circle (2pt);
\fill (2,2) circle (2pt);
\node at (2-.2, 2+.1) [above] {$M_{\nu}$};
\node at (2.2,-2.5-.1) [below] {$M_{\nu+1}$};
\end{tikzpicture}
\caption{$\triangle M_{\nu} P_{\nu+1} M_{\nu+1}$, $\varGamma_{P_{\nu+1}}$,  and $U_{P_{\nu+1}, \rho}$}\label{fig:atPnu}
\end{figure}
Note  that,
since this  $q_{\rho}$-modification 
does not affect the restriction 
of $\delta\sprime$ to a small neighborhood of the line segment 
$\gamma_{\nu+1}$ in $ C$, 
the function $\delta$ restricted to a small neighborhood of $\gamma_{\nu+1}$
in $C$ is equal to the unique affine-linear map  
that maps  $M_{\nu}$,  $P_{\nu+1}$, $  M_{\nu+1}$ 
to  $\tau_{\nu}$, $o$, $\tau_{\nu+1}$,  respectively.
\subsection{Intersection of \texorpdfstring{$\pants(C)_{\vareep}$}{DeltaCeprime} and \texorpdfstring{$\pants(C)_{\varee}$}{DeltaCe}}\label{subsec:proof5intnumb1}
We prove that the $E$-displacement and the $C$-displacement thus constructed satisfy 
condition (EC1).
Recall that 
\[
C_{\varee}=\set{Q+\thei \delta(Q) \varee}{Q\in C}.
\]
Hence we have
\[
C\cap C_{\varee}:=\set{Q\in C}{\delta(Q)=0}=\{Q_0\}\;\sqcup\; \Vertexes  (C),
\]
where $\{Q_0\}=\delta_M\inv (o)$.
Let $Q_0\sprime\in \YR$ be the point such that $\beta(Q_0\sprime)=Q_0$.
We have 
\[
\betash C\cap (\betash C)_{\varee}=\{Q_0\sprime \}\;\;\sqcup\;\;  \bigsqcup_{P\in \Vertexes  (C)} J_{C, P}.
\]
Since $(\betash C)_{\vareep}\subset \betash C$ and $(\betash C)_{\vareep}\cap J_{C, P}=\emptyset$,
we obtain 
\[
(\betash C)_{\vareep} \cap (\betash C)_{\varee}=\{Q_0\sprime \}.
\]
Recall that $\pants(C)_{\vareep}=\phi\inv((\betash C)_{\vareep})$
and $\pants(C)_{\varee}=\phi\inv((\betash C)_{\varee})$,
each equipped with the standard orientation $\sigma_C$.
We put
\[
\phi_{\Delta, \vareep}:=\phi|\pants(C)_{\vareep}\colon \pants(C)_{\vareep} \to (\betash C)_{\vareep},
\;\;
\phi_{\Delta, \varee}:=\phi|\pants(C)_{\varee}\colon \pants(C)_{\varee} \to (\betash C)_{\varee}.
\]
Let $Q\sprime_{0, +}$ and $Q\sprime_{0, -}$ 
be the points of $X$ such that $\phi\inv(Q_0\sprime)=\{Q\sprime_{0, +}, Q\sprime_{0, -}\}$.
We then have 
\[
\pants(C)_{\vareep}\cap \pants(C)_{\varee}=\{Q\sprime_{0, +}, Q\sprime_{0, -}\}.
\]
Interchanging $Q\sprime_{0, +}$ and $Q\sprime_{0, -}$ if necessary,
we can assume that  the restrictions  $\phi_{\Delta, \vareep}$ and $\phi_{\Delta, \varee}$
of $\phi$ to $\pants(C)_{\vareep}$ and to $\pants(C)_{\varee}$ 
are both orientation-preserving at $Q\sprime_{0, +}$,
and that  they are both orientation-reversing at $Q\sprime_{0, -}$.
In particular,
the local intersection number 
$(\pants(C)_{\vareep}, \pants(C)_{\varee})_{Q\sprime_{0, +}}$ of $\pants(C)_{\vareep}$ and $\pants(C)_{\varee}$
at $Q\sprime_{0, +}$
is equal to their local intersection number $(\pants(C)_{\vareep}, \pants(C)_{\varee})_{Q\sprime_{0, -}}$
 at $Q\sprime_{0, -}$.
Since $\beta\circ \phi$ is a local isomorphism locally around 
each of the points $Q\sprime_{0, \pm}$ of $X$, 
we see that
\[
(\pants(C)_{\vareep}, \pants(C)_{\varee})_{Q\sprime_{0, +}}=
(\pants(C)_{\vareep}, \pants(C)_{\varee})_{Q\sprime_{0, -}}=
(C, C_{\varee})_{Q_0},
\]
where $(C, C_{\varee})_{Q_0}$ is the local intersection number  of $C$ and $C_{\varee}$ at $Q_0$.
We show that 
$(C, C_{\varee})_{Q_0}=-1$.
\par 
We choose a  coordinate system $(\xi, \eta)$ of $\AtC$ that is good at $Q_0$.
(See Section~\ref{subsec:goodsystem}). 
We then put
\[
(x, y)=(\Real \xi, \Real \eta),
\quad 
(u, v)=(\Imag \xi, \Imag \eta).
\]
Then an ordered basis of the real tangent space at $Q_0$
of $\AtC$ oriented by the complex structure is 
\begin{equation}\label{eq:Derxuyv}
\Der{x}, \;\; \Der{u}, \;\; \Der{y},\;\; \Der{v}.
\end{equation}
We write  tangent vectors  at $Q_0$
in terms of this basis.
The tangent space at $Q_0$ of $C$ oriented by $\sigmaA$
has the ordered basis
\begin{equation*}\label{eq:a1a2}
\vect{a}_1:=(1,0,0,0), \;\; \vect{a}_2:=(0,0,1,0).
\end{equation*}
Recall that $C_{\varee}$ is given by the equations 
\[
u=\varee \delta_x(x, y), \quad v=\varee \delta_y(x, y), 
\]
where $\delta_x$ and $\delta_y$ are functions defined by~\eqref{eq:deltaxdeltay},
and that $C_{\varee}$ is 
oriented by $\sigmaA$
via the homeomorphism $\prR|C_{\varee}\colon C_{\varee}\cong C$.
Therefore 
the tangent space at $Q_0$ of $C_{\varee}$ with this orientation 
has the ordered basis
\begin{equation*}\label{eq:a3a4}
\vect{a}_3:=
\left(1,\,\varee\fDDer{\delta_x}{x} \,(Q_0),\, 0, \,\varee\fDDer{\delta_y}{x}\,  (Q_0)\right), \;
\vect{a}_4:=
\left(0,\,\varee\fDDer{\delta_x}{y} \,(Q_0),\, 1, \,\varee\fDDer{\delta_y}{y}\,  (Q_0)\right).
\end{equation*}
To prove $(C, C_{\varee})_{Q_0}=-1$,
it is enough to show $\det A<0$,
where  $A$ is the $4\times 4$ matrix 
whose $i$th row vector is $\vect{a}_i$ defined above for $i=1, \dots, 4$.
(See~\cite[Chapter~3]{GuilleminPollack1974}.)
Since $\delta$ is orientation-preserving at $Q_0$, its Jacobian  at $Q_0$ is $>0$:
\[
\fDDer{(\delta_x, \delta_y)}{(x, y)}\,  (Q_0) \;\;>\;\;0. 
\]
Therefore we have
\[
\det A=-\varee^2 \fDDer{(\delta_x, \delta_y)}{(x, y)}\,  (Q_0)<0.
\]
\subsection{Intersection of \texorpdfstring{$(H_{C, P})_{\vareep}$}{HCPeprime} and \texorpdfstring{$\pants(C)_{\varee} $}{DeltaCe}}
\label{subsec:proof5intnumb2}
Let $P$ be a vertex of $C$.
We investigate  the intersection of $(H_{C, P})_{\vareep}$ and $\pants(C)_{\varee} $,
and show that condition (EC2) is satisfied.
Recall that we have fixed good  local coordinate systems 
$(\xi_P, \eta_P)$, $(x_P, y_P)$, $(\tilxi_P, \mu_P)$,  $(\tilx_P, m_P)$
at $P$ (see Definition~\ref{def:goodsystems})
such that $a_P=-1$ and $ b_P=1$ hold (see~\eqref{eq:aPbPfixed}), 
and that the location of $C$ is Case~(1) of~\eqref{eq:Ccases}.
For simplicity,
we omit the subscript $P$ from these coordinates.
Let $\nu+1\in \ZZ/n\ZZ$ be the index of $P$.
By~\eqref{eq:aPbPfixed},
there exist positive real numbers 
$r, r\sprime, r\spprime, r\sppprime$ such that
\[
M_{\nu}=r(1,1), \;\;
M_{\nu+1}=r\sprime (1, -1),
\;\;
\tau_{\nu} =-r\spprime(e_x+e_y),
\;\;
\tau_{\nu+1} =r\sppprime(e_x-e_y).
\]
in terms of $(x, y)$, 
where  $e_x, e_y$ is the basis of $T(\AtR)$ given by~\eqref{eq:exey}.
Recall that,
in a small neighborhood $U_{\gamma}$ of 
the line-segment $\gamma_{\nu+1}=\{x=c\vareep\}\cap C$
in $C$, the  function $\delta$ is equal to the affine-linear map such that
$\delta(P)=o$, $\delta(M_{\nu})=\tau_{\nu}$, $\delta(M_{\nu+1})=\tau_{\nu+1}$.
Hence there exist real numbers $g, h$ 
satisfying
\begin{equation}\label{eq:gh}
-h<g<h
\end{equation}
such that, on $U_{\gamma}$,  we have
\begin{equation}\label{eq:deltagh}
\delta_x(x, y)=g x-h y,
\quad
\delta_y(x, y)=-h x+g y,
\end{equation}
where $\delta_x$ and $\delta_y$ are defined by~\eqref{eq:deltaxdeltay}.
Then,
by the formula~\eqref{eq:dspbcirc2},
the $C$-displacement $d\spsh_{\varee}\colon \betash C\to  (\betash C)_{\varee}$
associated with $\delta$ 
 is given
in a small neighborhood $\beta\inv(U_{\gamma})$ of $\beta\inv (\gamma_{\nu+1})$ in $\betash C$ by
\[
\begin{split}
&(\tilx, m)
\;\mapsto\;  \\
&\phantom{aa}(\tilxi, \mu)\;=\;\left(\,\tilx (1+\thei \varee (g-hm)),\;
\frac{m+\thei \varee (-h+g m)}{1+\thei \varee (g-hm)}\,\right)\;\in\;(\betash C)_{\varee}.
\end{split}
\]
Recall that the $E$-displacement
$(E_P)_{\vareep}$ is given by $\tilxi=c\vareep$
in a small neighborhood of $\beta\inv (\gamma_{\nu+1})$ in $\YC$.
Since $c\vareep\in \RR_{>0}$,
it follows that $(E_P)_{\vareep}$ and  $(\betash C)_{\varee}$ intersect only at one point,
which we write as
\[
R=d\spsh_{\varee}(R_0), 
\]
where $R_0$ is the point
\[
(\tilx, m)=(c\vareep, g/h).
\]
The coordinates of the point $R$  in terms of $(\tilxi, \mu)$ is 
\[
(\tilxi_R, \mu_R)=\left(c\vareep, \;\;\frac{g}{h}+\thei\varee\left(-h+\frac{g^2}{h}\right)\;\right).
\]
By~\eqref{eq:gh} and $\varee$ being very small,
the value $\mu_R$ of $\mu$ at $R$  
satisfies
\begin{equation}\label{eq:muR}
-1<\Real \mu_R<1,
\quad
\Imag \mu_R<0,
\quad
|\Imag \mu_R| \ll 1.
\end{equation}
More precisely,
we can make $|\Imag \mu_R| $ arbitrarily small by choosing a sufficiently small $\varee$.
Then $(D_P)_{\vareep}=\phi\inv ((E_P)_{\vareep})$
and $\pants(C)_{\varee}$ intersect at two points in $\phi\inv(R)$.
Let $\zeta$ be the function on $X$ defined by~\eqref{eq:zetasq}.
In the current situation,
we have
\[
\zeta^2=\frac{\mu+1}{\mu-1}.
\] 
We calculate the value of $\zeta$ at the points of $\phi\inv(R)$.
Note that, by~\eqref{eq:muR}, 
making $\varee$ smaller if necessary, 
we can assume that 
\[
\Real \frac{\mu_R+1}{\mu_R-1}<0, 
\quad 0<\Imag \frac{\mu_R+1}{\mu_R-1}\ll1.
\]
Hence 
one solution $\zeta_R$ of 
the equation  $\zeta^2=(\mu_R+1)/(\mu_R-1)$
satisfies
\begin{equation}\label{eq:zetaR}
0<\Real \zeta_R\ll1 ,  \quad 0<\Imag \zeta_R,
\end{equation}
and the other solution is $-\zeta_R$.
Let $\tilR_+$ and $\tilR_-$
be the points of $\phi\inv(R)$ such that   $\zeta(\tilR_+)=\zeta_R$
and $\zeta(\tilR_-)=-\zeta_R$.
\par
Recall from Table~\ref{table:hemispheres} that,
because we are in Case (1) of~\eqref{eq:Ccases},  the circle $S_{C, P}$ is given by $\Real\zeta=0$ on 
$D_P=\{\phi^*\tilxi=0\}$.
Hence one and only one of $\tilR_+$ or $ \tilR_-$ belongs to $(H_{C, P})_{\vareep}$.
In particular,
the set-theoretic intersection of $(H_{C, P})_{\vareep}$ and $(\pants(C))_{\varee}$ consists of a single point.
We will show that the local intersection number at this point is $1$.
\par
Suppose that $\tilR_+\in (H_{C, P})_{\vareep}$.
By~\eqref{eq:zetaR}, the hemisphere  $H_{C, P}$ is given by $\Real\zeta|D_P\ge 0$,
that is, we are in the second row of Table~\ref{table:functions}.
The orientation $\ori_H$ on $S_{C, P}=\bdr H_{C, P} $ given by the complex structure of $H_{C, P}$
is $\downarrow$.  
Hence the orientation $\ori_C$ on $S_{C, P}=\bdr \pants(C) $ given by 
the standard orientation on $\pants(C) $ is $\uparrow$.
Since
$\phi$ is given by 
\[
\mu=\frac{\zeta^2+1}{\zeta^2-1},
\]
and since $\Image \zeta_R>0$ by~\eqref{eq:zetaR},
the image by $\phi$ of the orientation $\ori_C$ of $S_{C, P}=\bdr \pants(C) $
near $\tilR_+$ is the $m$-increasing direction on $J_{C, P}$.
Indeed, when $\zeta$ moves from $0$ to $\infty \thei $ along $\thei \RR_{\ge 0}$,
then $\mu$ moves from $-1$ to $1$ along $\RR$.
Recall that the standard orientation of $\betash C$ induces 
the $m$-decreasing orientation on $J_{C, P}$.
(See Corollary~\ref{cor:downward} and  Figure~\ref{fig:betashC}.)
Therefore $\tilR_+$ is on the \emph{negative}-sheet of $(\pants(C))_{\varee}$.
\par 
Suppose that $\tilR_-\in (H_{C, P})_{\vareep}$.
By~\eqref{eq:zetaR}, we see that $H_{C, P}$ is given by $\Real\zeta\le 0$, 
and hence we are in the first row of Table~\ref{table:functions}.
The orientation $\ori_H$ on $S_{C, P}=\bdr H_{C, P} $  
is $\uparrow$.  
Hence the orientation $\ori_C$ on $S_{C, P}=\bdr \pants(C) $ 
is  $\downarrow$.
Since
 $-\Image \zeta_R<0$ by~\eqref{eq:zetaR},
the image by $\phi$ of the orientation $\ori_C$ of $S_{C, P}=\bdr \pants(C) $
near $\tilR_-$ is the $m$-increasing orientation on $J_{C, P}$.
Therefore $\tilR_-$ is also on the \emph{negative}-sheet of $(\pants(C))_{\varee}$.
\par
In any case,
the restriction of  $\beta\circ \phi$ to a neighborhood in $\pants(C)_{\varee}$ of
the intersection point $\tilR_+$ or $\tilR_-$ of $(H_{C, P})_{\vareep}$ and  $\pants(C)_{\varee}$
is an orientation-\emph{reversing} local  isomorphism to 
an open neighborhood of $\beta(R)$ in $C_{\varee}$.
Hence, to show that the local intersection number of
$(H_{C, P})_{\vareep}$ and $(\pants(C))_{\varee}$
 is $1$,
it is enough to prove that,
in $\AtC$,
the local intersection number at $\beta(R)$ of
\[
\beta((E_P)_{\vareep})=\{\xi=c\vareep\}
\]
with the orientation coming from the complex structure and the space $C_{\varee}$
with the orientation induced by $\sigmaA$ is $-1$.
Note that
$C_{\varee}$ is 
given locally around $\beta(R)$ by 
\begin{align*}
\xi&=x+\thei \varee \delta_x (x, y)=x+\thei \varee (gx-hy),\\
\eta&=y+\thei \varee \delta_y (x, y)=y+\thei \varee (-hx+gy).
\end{align*}
We use the ordered basis~\eqref{eq:Derxuyv} of the real tangent space of $\AtC$.
Then the   tangent space of $\beta((E_P)_{\vareep})$ 
oriented by the complex structure
has an ordered basis
\[
\vect{b}_1:=(0,0,1,0),
\quad
\vect{b}_2:=(0,0,0,1).
\]
On the other hand,   the oriented  tangent space of $C_{\varee}$ has an ordered basis 
\begin{eqnarray*}
\vect{b}_3&:=&\left(1, \varee\fDDer{\delta_x}{x}, 0,  \varee\fDDer{\delta_y}{x}\right)=(1,  \varee g, 0,  -\varee h ),\\
\vect{b}_4&:=&(0, \varee\fDDer{\delta_x}{y},1, \varee\fDDer{\delta_y}{y})=(0, - \varee h, 1,  \varee g ).
\end{eqnarray*}
Let   $B$ be the $4\times 4$ matrix 
whose $i$th row vector is $\vect{b}_i$ defined above for $i=1, \dots, 4$.
Since $h>0$ by~\eqref{eq:gh}, 
we have $\det B=-\varee  h<0$.
Therefore the local intersection number at $\beta(R)$ of
$\beta((E_P)_{\vareep})$ and $C_{\varee}$  is $-1$.
(See~\cite[Chapter 3]{GuilleminPollack1974}.)
Thus  (EC2) holds at $P$.
 \qed
\section{Examples}\label{sec:examples}
Using a smooth projective completion  $\tilX$ of $X$, 
we  check certain consequences of Theorem~\ref{thm:intnumbs}  in several  examples.
%
\subsection{Lattice}\label{subsec:lattice}
A \emph{lattice} is a free $\ZZ$-module $L$ of finite rank
with a \emph{non-degenerate} symmetric bilinear form $\intfnull\colon L\times L\to \ZZ$.
Let $L$ be a lattice.
The \emph{signature} of $L$ is the pair $(s_+, s_-)$
of numbers  of positive and negative eigenvalues of the Gram matrix of $L$.
A lattice $L$ is embedded naturally into its dual $L\dual:=\Hom(L, \ZZ)$.
The \emph{discriminant group} of $L$ is the finite abelian group 
\[
\discg (L):=L\dual/L.
\]
We say that $L$ is \emph{unimodular} if its discriminant group is trivial.
By definition, we have the following:
\begin{proposition}\label{prop:subquotient}
Let $L\sprime$ be a sublattice of a lattice $L$ with a finite index $m$.
Then $\disc(L)$ is a sub-quotient of $\disc(L\sprime)$, and 
$|\disc(L)|$ is equal to $|\disc(L\sprime)|/m^2$.
\qed
\end{proposition}
The following is well known. 
See, for example,~\cite{theNikulin}.
\begin{proposition}\label{prop:discg}
Let $L$ be a sublattice of a unimodular lattice $M$.
Suppose that $M/L$ is torsion-free.
We put
\begin{equation}\label{eq:Lsperp}
L\sperp:=\set{x\in M}{\intf{x, y}=0\;\;\textrm{for all}\;\; y\in L}.
\end{equation}
Then we have $\discg (L) \cong \discg (L\sperp)$.
\qed
\end{proposition}
Combining Propositions~\ref{prop:subquotient}~and~\ref{prop:discg}, we obtain the following:
\begin{corollary}\label{cor:subquot}
Let $L$ be a sublattice of a unimodular lattice $M$.
Then $\discg (L\sperp)$ is a sub-quotient of  $\discg (L)$.
\qed
\end{corollary}
Let $\Arr$ be a nodal real line arrangement,
and let $\intfnull$ be the intersection form on $H_2(X; \ZZ)$.
We put 
\begin{equation}\label{eq:Ker}
\Ker \,\intfnull:=\set{x\in H_2(X; \ZZ)}{\intf{x, y}=0 \;\;\textrm{for any}\;\; y\in H_2(X; \ZZ)}.
\end{equation}
Then $\intfnull$ yields a natural structure of the lattice on 
\[
\barH(X):=H_2(X; \ZZ)/\Ker \intfnull,
\]
and we can calculate the  Gram matrix of $\barH(X)$ by Theorem~\ref{thm:intnumbs}.
In particular, we can calculate the signature and the discriminant group of $\barH(X)$ by Theorem~\ref{thm:intnumbs}.
\subsection{A projective completion of \texorpdfstring{$X$}{X}}\label{subsec:completion}
Let $\PtC$ be the complex projective plane containing $\AtC$
as an affine part.
We put
\[
\ellinf:=\PtC\setminus \AtC.
\]
For $\ell_i(\CC)\in \ArrC$,
we denote by $\tilell_i(\CC)$ the projective completion of $\ell_i(\CC)$ in $\PtC$,
and put
\[
\tilArrC:=
\begin{cases}
\set{\tilell_i(\CC)}{\ell_i(\CC)\in \ArrC}\;\cup\;\{\ellinf\} & \textrm{if $|\Arr|$ is odd}, \\
\set{\tilell_i(\CC)}{\ell_i(\CC)\in \ArrC} & \textrm{if $|\Arr|$ is even.}
\end{cases}
\]
Let $\tilBC$ be the union of the complex projective lines in $\tilArrC$.
Then $\deg \tilBC$ is even.
We consider the morphisms 
\begin{equation}\label{eq:rhopiproj}
\tilX\maprightsp{\tilrho} \tilW \maprightsp{\tilpi} \PtC,
\end{equation}
where  $\tilpi\colon \tilW\to \PtC$ is the double covering whose branch locus is equal to $\tilBC$,
and $\tilrho\colon \tilX\to \tilW$ is the minimal desingularization.
We put
\[
\Lambda_{\infty}:=\tilrho\inv (\tilpi\inv(\ellinf))\;\;\subset\;\; \tilX.
\]
Then we have $X=\tilX\setminus \Lambda_{\infty}$, and 
the inclusion $\iota\colon X\inj \tilX$ induces  natural homomorphisms
\[
\iota_{*}\;\colon\; H_2(X; \ZZ)\to H_2(\tilX; \ZZ) \to H_2(\tilX; \ZZ)_{\free}:=H_2(\tilX; \ZZ)/\mathrm{torsion},
\]
which preserves the intersection form.
Since $\tilX$ is smooth and compact,
the intersection form makes the free $\ZZ$-module $H_2(\tilX; \ZZ)_{\free}$ a unimodular lattice.
Let 
\[
\Hinf\subset H_2(\tilX; \ZZ)_{\free}
\]
 be the submodule generated by 
the classes of  irreducible components of $\Lambda_{\infty}$.
Then the image of $\iota_{*}$
is equal to the orthogonal complement $\Hinf\sperp$ of $\Hinf$ in $H_2(\tilX; \ZZ)_{\free}$ defined by~\eqref{eq:Lsperp}.
Suppose that $\Hinf$ is a sublattice of $H_2(\tilX; \ZZ)$,
that is, the intersection form restricted to  $\Hinf$ is non-degenerate.
Then the intersection form restricted to the image $\Hinf$ of $\iota_{*}$ is also non-degenerate, and hence 
the kernel of $\iota_{*}$ is  equal to $\Ker \, \intfnull$
defined by~\eqref{eq:Ker}.
Therefore the two lattices $\barH(X)$ and $\Hinf\sperp$  are isomorphic:
\begin{equation}\label{eq:HH}
\barH(X)\;\; \cong\;\;  \Hinf\sperp.
\end{equation}
\subsection{The lattices \texorpdfstring{$H_2(\tilX; \ZZ)$}{H2Xtilde} and \texorpdfstring{$\Hinf\sperp$}{Hinftyperp}}
In this section, 
we consider  arrangements $\Arr$ with the following property:
\begin{equation}
\label{eq:parallelassumption}
\parbox{11.5cm}{\;\;
For any   $\lambda_i(\RR)\in \Arr$, at most one line in  $\Arr\setminus\{\lambda_i(\RR)\}$ is  parallel to $\lambda_i(\RR)$.}
\end{equation}
We put $\nopslines:=|\Arr|$, and assume $\nopslines\ge 3$. We then put 
\[
\tilN:=|\tilArrC|=\begin{cases}
\nopslines+1 & \textrm{if $\nopslines$ is odd, }\\
\nopslines & \textrm{if $\nopslines$ is even. }
\end{cases}
\]
Suppose that $\Arr$ has exactly $p$ parallel pairs, where $2p\le \nopslines$.
Then we have
\begin{equation}\label{eq:sizees}
|\Chamb|=\frac{(\nopslines-1)(\nopslines-2)}{2}-p,
\quad 
|\SingBC|=\frac{\nopslines(\nopslines-1)}{2}-p.
\end{equation}
\begin{proposition}\label{prop:Hinf}
Suppose that~\eqref{eq:parallelassumption} holds and that $\nopslines\ge 3$.
Let $p$ be the number of parallel pairs in $\Arr$.
\begin{enumerate}[{\rm (1)}]
\item
The homology group $H_2(\tilX; \ZZ)$ is torsion-free, and 
the  unimodular lattice 
$H_2(\tilX; \ZZ)$  is of rank $\tilN^2-3\tilN +4$ with signature 
\begin{equation*}\label{eq:unimod}
\left(\frac{1}{\,4\,}\tilN^2-\frac{3}{\,2\,}\tilN +3, \; \frac{3}{\,4\,}\tilN^2-\frac{3}{\,2\,}\tilN +1\right).
\end{equation*}
\item
The submodule $\Hinf$  of $H_2(\tilX; \ZZ)$ is  a sublattice of rank
\[
\rinf :=\rank \Hinf=\begin{cases}
 1+\nopslines+2p & \textrm{if $\nopslines$ is odd}, \\
 1+p & \textrm{if $\nopslines$ is even and $\nopslines\ne 2p$}, \\
2+p & \textrm{if $\nopslines=2p$},  \\
\end{cases}
\]
with the signature $(1, \rinf-1)$.
The discriminant group  is given by 
\[
\disc (\Hinf)\cong
\begin{cases}
(\ZZ/2\ZZ)^{\nopslines-1} & \textrm{if $\nopslines$ is odd}, \\
(\ZZ/2\ZZ)^{p+1} & \textrm{if $\nopslines$ is even and $\nopslines\ne 2p$}, \\
(\ZZ/2\ZZ)^{p-1} \times (\ZZ/2(p-1)\ZZ)& \textrm{if $\nopslines=2p$}. \\
\end{cases}
\]
\end{enumerate}
\end{proposition}
By~Theorem~\ref{thm:basis},  we have
\[
\rank H_2(X; \ZZ)=|\Chamb|+|\SingBC|=(\nopslines-1)^2-2p.
\]
Suppose that $\nopslines$ is odd $\ge 3$.
By Proposition~\ref{prop:Hinf}, we have 
\[
\rank H_2(\tilX; \ZZ)-\rank \Hinf=\tilN^2-3\tilN+4-(1+N+2p).
\]
Since these two are equal, 
the equality~\eqref{eq:HH} implies  $\Ker \,\intfnull=0$. 
Therefore we obtain the following:
\begin{corollary}
Suppose that~\eqref{eq:parallelassumption} holds and that  $\nopslines$ is odd $\ge 3$.
Then  the intersection form $\intfnull$ on $H_2(X; \ZZ)$ is non-degenerate.
\qed
\end{corollary}
\begin{proof}[Proof of Proposition~\ref{prop:Hinf}]
Let $\PPPinf$ denote the set of 
singular points of $\tilBC$ located on $\ellinf$.
By assumption~\eqref{eq:parallelassumption}, 
it follows that $\PPPinf$ consists of simple singular points of type $a_1$ (ordinary double points)
and/or  of type $d_4$ (ordinary triple points).
Let 
\[
\PPPinf=\PPPinfa\sqcup \PPPinfd
\]
be the decomposition of $\PPPinf$ according to the types.
We have the following.
\begin{itemize}
\item If $\nopslines$ is odd, then $|\PPPinfa|=\nopslines-2p$ and $|\PPPinfd|=p$.  
\item If $\nopslines$ is even, then $|\PPPinfa|=p$ and $|\PPPinfd|=0$.  
\end{itemize}
In particular, the branch curve  $\tilBC$ 
of the double covering $\tilpi\colon \tilW\to \PtC$ 
has only simple singularities,  and hence $\tilW$ has only rational double points 
as its singularities.
By the theory of simultaneous resolution of rational double points of algebraic surfaces,
we see that $\tilX$ is diffeomorphic to the double cover $\tilX\sprime$ of $\PtC$ branching along 
a \emph{smooth} projective plane curve of degree $\tilN$.
Since $H_2(\tilX\sprime; \ZZ)$ is torsion-free, so is $H_2(\tilX; \ZZ)$.
We can calculate the second Betti number $b_2$ and 
the geometric genus $p_g$ of $\tilX\sprime$ easily.
Note that the signature of $H_2(\tilX\sprime; \ZZ)$ is
\[
(1+2p_g, b_2-1-2p_g).
\]
Therefore we obtain a proof of assertion (1).
\par
For $P\in \PPPinfa$,
let $a(P)$ denote the 
smooth rational curve on $\tilX$ that is contracted to $P$.
We have $\intf{a(P), a(P)}=-2$.
For $Q\in \PPPinfd$,
there exist four smooth rational curves 
\[
d_1(Q), \;\; d_2(Q),\;\; d_3(Q),\;\; d_4(Q)
\]
on $\tilX$ that are contracted to $Q$.
These curves have the self-intersection  $(-2)$,
and form the Coxeter-Dynkin diagram  of type $D_4$ as the dual graph.
Let $\Rinf$ denote  the reduced part of the strict transform of $\ellinf$ in $\tilX$.
Note that  $\Rinf$ is irreducible if and only if $\nopslines\ne 2p$.
Then $\Hinf$ is generated by 
\begin{itemize}
\item[(a)]  the classes of $a(P)$, where $P$ runs though $\PPPinfa$,
\item[(d)] the classes of $d_1(Q), \dots, d_4(Q)$, where   $Q$ runs though $\PPPinfd$,
and 
\item[(R)]  the classes of irreducible components of $\Rinf$.
\end{itemize}
Let $\Hinf\sprime$ be the submodule of $\Hinf$ generated by the classes  in (a) and (d) above, and 
\begin{itemize}
\item[(L)]  the class of the total transform $\LLLinf$ of $\ellinf$ in $\tilX$.
\end{itemize}
Since we have $\intf{\LLLinf, \LLLinf}=2$ and $\LLLinf$ is orthogonal to the classes of exceptional curves in (a) and (d) above, 
it follows  that $\Hinf\sprime$ is a lattice with 
\[
\rank \Hinf\sprime=1+|\PPPinfa|+4|\PPPinfd|,
\quad
\disc \Hinf\sprime\cong (\ZZ/2\ZZ)^{1+|\PPPinfa|+2|\PPPinfd|}.
\]
\par
Suppose that $\nopslines$ is odd.
Then $\Rinf$ is irreducible
and the  multiplicity of $\Rinf$ in $\LLLinf$ is $2$.
Hence $ \Hinf\sprime$ is a submodule of index $2$ in $ \Hinf$.
Therefore $ \Hinf$ is a lattice,
and the rank and the discriminant group of $ \Hinf$ can be derived from those of $ \Hinf\sprime$.
\par
Suppose that $\nopslines$ is even and $2p<\nopslines$.
Then $\Rinf$ is irreducible
and the  multiplicity of $\Rinf$ in $\LLLinf$ is $1$.
Hence we have $ \Hinf\sprime= \Hinf$.
\par
Suppose that $\nopslines=2p$.
Then $\Rinf$ is a disjoint union of two smooth rational curves
$\Rinfa$ and $\Rinfb$.
For $P\in \PPPinfa$,
we have $\intf{a(P), \Rinfa}=\intf{a(P), \Rinfb}=1$.
Using 
\[
\LLLinf=\Rinfa+\Rinfb+\sum_{P\in \PPPinfa} a(P)
\]
and  $\intf{\LLLinf, \LLLinf}=2$,  we obtain 
\[
\intf{\Rinfa, \Rinfa}=\intf{\Rinfb, \Rinfb}=1-p.
\]
By these formulas, we see that $\Hinf$ is a lattice of rank $2+|\PPPinfa|$.
From the Gram matrix of $\Hinf$, we can compute the discriminant group.
The computation is left to the reader.
%
\par
Since $\Hinf$ is contained in the Hodge part $H^{1,1}(\tilX)$ of $H^{2}(\tilX; \CC)$ and $\LLLinf\in \Hinf$ is an ample class,
we see that the signature of $\Hinf$ is $(1, \rinf-1)$.
\end{proof}
\subsection{Experiments}
For fixed $\nopslines\ge 3$ and $p\le \nopslines/2 $,
we randomly generate a nodal real line arrangement $\Arr$ satisfying~\eqref{eq:parallelassumption}, 
and  calculate the signature and the discriminant group of $ \barH(X)$ 
topologically  by Theorem~\ref{thm:intnumbs}.
We compare the result with  the signature  and  the discriminant group of $\Hinf$
computed by Proposition~\ref{prop:Hinf}. 
We see that the equality of signatures holds for $\barH(X)$ and $\Hinf\sperp$, as is expected from~\eqref{eq:HH}, 
and that $ \disc(\barH(X))$ is a sub-quotient of $ \disc(\Hinf)$, as is expected by~\eqref{eq:HH} and  Corollary~\ref{cor:subquot}.
\begin{remark}
In fact, 
for all examples we computed, we obtain an isomorphism $ \disc(\barH(X))\cong \disc(\Hinf)$,
and hence, in these cases,  the sublattice $\Hinf$ is primitive in $H_2(\tilX; \ZZ)$.
\end{remark}
\begin{example}
Suppose that $\nopslines=6$ and $p=0$.
We have 
\[
|\Chamb|=10,
\quad
|\SingBC|=15, 
\]
and hence $H_2(X; \ZZ)$ is of rank $25$.
On the other hand, the unimodular lattice $H_2(\tilX; \ZZ)$ is of rank $22$ with signature $(3, 19)$, and 
the sublattice  $H_{\infty}$ is of rank $1$ with signature $(1, 0)$
and   $\disc (H_{\infty})\cong \ZZ/2\ZZ$.
Hence $H_{\infty}\sperp$ 
is of rank $21$ with signature $(2, 19)$.
\par
For randomly generated arrangements satisfying~\eqref{eq:parallelassumption}
with $\nopslines=6$ and $p=0$, we checked that 
$\barH(X)$ is of rank $21$ with signature $(2, 19)$
and $\disc (\barH(X))\cong \ZZ/2\ZZ$.
Remark that 
there exist several combinatorial structures of nodal arrangements of six real lines with no parallel pairs.
For example, 
the numbers of $n$-gons in $\Chamb$ can vary as follows:
\[
\begin{array}{c|cccc} 
n & 3 & 4 & 5 & 6  \\ 
 \hline
 & 4 & 4 & 2 & 0 \\ 
 & 4 & 5 & 1 & 0 \\ 
 & 4 & 5 & 0 & 1 \\ 
 & 4 & 6 & 0 & 0 \\ 
 & 5 & 3 & 2 & 0 \\ 
 & 5 & 4 & 1 & 0 
 \end{array}
\qquad
\begin{array}{c|cccc} 
n & 3 & 4 & 5 & 6  \\ 
 \hline
 & 5 & 4 & 0 & 1 \\ 
 & 6 & 2 & 2 & 0 \\ 
 & 6 & 3 & 1 & 0 \\ 
 & 6 & 3 & 0 & 1 \\ 
 & 7 & 0 & 3 & 0\rlap{\;\;.} \\
 &   &    &    &
\end{array}
\]
\end{example}
\begin{example}
Suppose that $\nopslines=6$ and $p=3$.
We have 
\[
|\Chamb|=7,
\quad
|\SingBC|=12, 
\]
and hence $H_2(X; \ZZ)$ is of rank $19$.
On the other hand,
the unimodular lattice $H_2(\tilX; \ZZ)$ is of rank $22$ with signature $(3, 19)$, and 
 the lattice $H_{\infty}$ is of rank $5$ with signature $(1, 4)$
and   $\disc (H_{\infty})\cong ( \ZZ/2 \ZZ)^2\times ( \ZZ/4 \ZZ)$.
Hence $H_{\infty}\sperp$ 
is of rank $17$ with signature $(2, 15)$.
\par
For randomly generated such arrangements, we checked that 
$\barH(X)$ is of rank $17$ with signature $(2, 15)$
and  $\disc (\barH(X))\cong ( \ZZ/2 \ZZ)^2\times ( \ZZ/4 \ZZ)$,
regardless of combinatorial structures of $\Chamb$.
\end{example}
\begin{example}
Suppose that $\nopslines=24$ and $p=10$.
We have 
\[
|\Chamb|=243,
\quad
|\SingBC|=266, 
\]
and hence $H_2(X; \ZZ)$ is of rank $509$.
On the other hand,
the unimodular lattice $H_2(\tilX; \ZZ)$ is of rank $508$ with signature $(111, 397)$, and 
 the sublattice $H_{\infty}$ is of rank $11$ with signature $(1, 10)$
and  $\disc(H_{\infty})\cong ( \ZZ/2 \ZZ)^{11}$.
Hence $H_{\infty}\sperp$ 
is of rank $497$ with signature $(110, 387)$.
For randomly generated such arrangements, we obtained expected signature $(110, 387)$,
and  $\disc(\barH(X))\cong ( \ZZ/2 \ZZ)^{11}$.
\end{example}
\subsection*{Acknowledgment}
We thank Professor Masahiko Yoshinaga for his valuable discussions.
We also thank Professor Eugenii Shustin for introducing us to~\cite{Shustin1990}, 
and the referees for their many helpful comments and suggestions on the first version of this paper.
%
%

\end{document}